\documentclass[12pt,draft,reqno]{amsart}

\usepackage[cp1251]{inputenc}
\usepackage[T2A]{fontenc}
\usepackage[ukrainian,english]{babel}
\usepackage{amsmath,amsfonts,amssymb}
\usepackage{geometry}
\usepackage{cite}
\newtheorem{theorem}{Theorem}[section]
\newtheorem{corollary}[theorem]{Corollary}
\newtheorem{lemma}[theorem]{Lemma}
\newtheorem{proposition}[theorem]{Proposition}
\theoremstyle{definition}

\newtheorem{remark}[theorem]{Remark}
\newtheorem{condition}[theorem]{Condition} 

\numberwithin{equation}{section}

\author[V. Los, V. Mikhailets, A. Murach]
{Valerii Los, Vladimir  Mikhailets, Aleksandr  Murach}

\title[Parabolic problems in generalized Sobolev spaces]
{Parabolic problems in generalized Sobolev spaces}

\address{National Technical University of Ukraine "Igor Sikorsky Kyiv Polytechnic Institute", Prospect Peremohy 37, 03056, Kyiv-56, Ukraine}

\email{v\_los@yahoo.com}

\address{Institute of Mathematics, National Academy of Sciences of Ukraine, 3 Tereshchenkivs'ka, Kyiv, 01004, Ukraine}

\email{mikhailets@imath.kiev.ua}

\address{Institute of Mathematics, National Academy of Sciences of Ukraine, 3 Tereshchenkivs'ka, Kyiv, 01004, Ukraine}

\email{murach@imath.kiev.ua}

\subjclass[2000]{Primary: 35K35; Secondary: 46E35.}

\keywords{Parabolic problem, generalized Sobolev space,
H\"ormander space, slowly varying function, isomorphism property, local regularity of solution, interpolation with function parameter.}

\thanks{The publication contains the results of studies conducted by the joint grant F81 of the National Research Fund of Ukraine and the German Research Society (DFG); competitive project F81/41686.}

\begin{document}

\maketitle

\begin{abstract}
We consider a general inhomogeneous parabolic initial-boundary value problem for a $2b$-parabolic differential equation given in a finite multidimensional cylinder. We investigate the solvability of this problem in some generalized anisotropic Sobolev spaces. They are parametrized with a pair of positive numbers $s$ and $s/(2b)$ and with a function $\varphi:[1,\infty)\to(0,\infty)$ that varies slowly at infinity. The function parameter $\varphi$ characterizes subordinate regularity of distributions with respect to the power regularity given by the number parameters. We prove that the operator corresponding to this problem is an isomorphism on appropriate pairs of these spaces. As an application, we give a theorem on the local regularity of the generalized solution to the problem. We also obtain sharp sufficient conditions under which chosen generalized derivatives of the solution are continuous on a given set.
\end{abstract}

\bigskip

\section{Introduction}\label{16sec1}

Isomorphism theorems form a core of the modern theory of parabolic initial-boundary value problems. These theorems assert that the parabolic problems are well-posed (in the sense of Hadamard) on appropriate pairs of Sobolev or H\"older anisotropic normed spaces \cite{AgranovichVishik64, DenkHieberPruss07, Friedman64, Eidelman64, Eidelman94, ZhitarashuEidelman98, LadyzhenskajaSolonnikovUraltzeva67, LionsMagenes72ii}. Otherwise speaking, the bounded operators corresponding to the mentioned problems set isomorphisms on these pairs. The isomorphism theorems play a key role in the investigations of regularity of solutions to parabolic problems, their Green functions, control problems for systems governed by parabolic equations, and others (see., e.g., \cite{Eidelman64, Eidelman94, LionsMagenes72ii, Ivasyshen90, Lunardi95}). Note that the case of Hilbert spaces is of a special interest for applications of these theorems \cite{Eidelman94, LionsMagenes72ii}). This case deals with anisotropic Sobolev spaces \cite{Slobodetskii58} based on the Lebesgue spaces of square integrable functions.

Certainly, the more finely a scale of function spaces is calibrated, the more precise results may be obtained with the help of this scale. Number parameters, which is used for Sobolev or H\"older spaces, provide a calibration that proved to be rough for various problems of mathematical analysis \cite{FarkasLeopold06, KalyabinLizorkin87, Stepanets05, Triebel01}, theory of partial differential equations \cite{Hermander63, Hermander83, MikhailetsMurach14, NicolaRodino10, Paneah00}, theory of stochastic processes \cite{Jacob010205}, and others. In 1963, motivated by applications to partial differential equations, H\"ormander \cite{Hermander63} introduced and investigated the normed distribution spaces
$$
\mathcal{B}_{p,\mu}:=\bigl\{w\in\mathcal{S}'(\mathbb{R}^{k})
:\mu\widehat{w}\in L_p(\mathbb{R}^{k})\bigr\}
$$
parametrized with a sufficiently general function parameter $\mu:\mathbb{R}^{k}\to(0,\infty)$. Here, the number parameter $p$ satisfies $1\leq p\leq\infty$, and $\widehat{w}$ denotes the Fourier transform of the tempered distribution~$w$. H\"ormander gave important applications of his spaces to the investigation of the existence and regularity of solutions to partial differential equations (see also his monograph \cite{Hermander83}). The most complete results were obtained for the class of hypoelliptic equations, to which parabolic equations pertain. If $p=2$, the H\"ormander spaces become Hilbert ones and turn out to be a broad generalization of the inner product Sobolev spaces.

H\"ormander's monograph \cite{Hermander63} attracted a great attention to generalized Sobolev spaces and stimulated various investigations concerning these spaces and their applications, mostly to mathematical analysis (see, e.g., \cite{FarkasLeopold06, KalyabinLizorkin87, Triebel01, VolevichPaneah65} and references therein). However, these spaces were applied rarely to multidimensional boundary-value problems up to the recent time (we may refer only to \cite{Slenzak74}). This was caused by the absence of reasonable notion of generalized Sobolev spaces over manifolds (such spaces should be independent of local charts on the manifold) and by the lack of analytical tools to work with these spaces.

Recently the situation has essentially changed. Mikhailets and Murach \cite{MikhailetsMurach12BJMA2, MikhailetsMurach14} have built a theory of solvability of general elliptic boundary-value problems in generalized Sobolev spaces of the form  $H^{s;\varphi}:=\mathcal{B}_{2,\mu}$, where
$$
\mu(\xi):=(1+|\xi|^{2})^{s/2}\varphi((1+|\xi|^{2})^{1/2}),
$$
$s\in\mathbb{R}$, and the function $\varphi:[1,\infty)\to(0,\infty)$ varies slowly at infinity in the sense of Karamata. Note that these spaces are isotropic because the function $\mu$ depends only on $|\xi|$. The function parameter $\varphi$ defines a subordinate regularity of the distributions $w\in H^{s;\varphi}$ with respect to the basic power regularity given by the number $s$. If $\varphi(\cdot)\equiv1$, the space $H^{s;\varphi}$ will become the inner product Sobolev space of order $s$. The main research method of this theory is the interpolation with a function parameter of Hilbert spaces and linear operators acting on these spaces. The H\"ormander spaces used in the theory are obtained by this interpolation applied to pairs of inner product Sobolev spaces. This allows the authors to define the corresponding spaces over smooth manifolds and facilitates the application of these spaces to elliptic problems. Of late years this theory was extended to a wider class of generalized Sobolev spaces \cite{AnopKasirenko16MFAT, AnopMurach14UMJ}, namely to all Hilbert spaces that are interpolation ones between inner product Sobolev spaces \cite{MikhailetsMurach13UMJ3, MikhailetsMurach15ResMath1}.

The above-mentioned interpolation method proved to be useful in the theory of parabolic initial-boundary value problems as well. This was shown in papers \cite{Los15UMG5, Los16UMG6, LosMikhailetsMurach17CPAA, LosMurach13MFAT2, LosMurach17OpenMath} for some classes of parabolic problems. These papers deal with the generalized anisotropic Sobolev spaces $H^{s,s/(2b);\varphi}:=\mathcal{B}_{2,\mu}$, where
$$
\mu(\xi',\xi_{k})=\bigl(1+|\xi'|^2+|\xi_{k}|^{1/b}\bigr)^{s/2}
\varphi\bigl((1+|\xi'|^2+|\xi_{k}|^{1/b})^{1/2}\bigr)
$$
for all $\xi'\in\mathbb{R}^{k-1}$ and $\xi_{k}\in\mathbb{R}$. Here, $s\in\mathbb{R}$, $\varphi$ is the above-mentioned function, and $\nobreak{1\leq b\in\mathbb{Z}}$, with the even integer $2b$ characterizing the parabolicity of the partial differential equation investigated. These papers present isomorphism theorems for the parabolic problems considered in the indicated spaces and give some applications of these theorems to the study of the regularity of generalized solutions to the parabolic problems.

Note that various methods of the interpolation with a number parameter between normed spaces are used in the theory of multidimensional boundary-value problems \cite{Berezansky68, LionsMagenes72i, LionsMagenes72ii, Roitberg96, Triebel83, Triebel95}. However, the application of these methods to Sobolev spaces (or other classical function spaces depending on number parameters only) does not give spaces parametrized with function parameters.

The purpose of this paper is to prove an isomorphism theorem for a \emph{general} parabolic initial-boundary value problem given in a finite multidimensional cylinder and considered in the generalized Sobolev spaces $H^{s,s/(2b);\varphi}$. We use this theorem to prove some results on the local regularity of generalized solutions to the problem under investigation. Our main method is the interpolation with a function parameter between Hilbert spaces. This method allows us to deduce the isomorphism theorem from the known theorem \cite{AgranovichVishik64, Eidelman94, ZhitarashuEidelman98, LionsMagenes72ii, Zhitarashu85} on the well-posedness of the general parabolic problem in anisotropic Sobolev spaces. Note that the transition from the case \cite{LosMikhailetsMurach17CPAA} of homogeneous initial conditions to the general case is not easy even for the Sobolev spaces of integer orders (see \cite[Sections 10 and 11]{AgranovichVishik64}). Specifically, this transition uses the description of the spaces in terms of the spacial variables and time variable. To avoid this difficulty and other obstacles, we prefer to resort to the interpolation of spaces and operators that correspond to the parabolic problem with inhomogeneous initial conditions.

The paper consists of six sections and Appendix. Section~\ref{16sec1} is Introduction. Section~\ref{16sec2} contains the statement of the general parabolic initial-boundary value problem. Section~\ref{16sec3} discusses generalized anisotropic Sobolev spaces in which we investigate this problem. Our main results are formulated in Section~\ref{16sec4}. The basic result is Isomorphism Theorem~\ref{16th4.1} for the parabolic problem considered in the above-mentioned generalized Sobolev spaces. As applications of this theorem, we give Theorems \ref{16th4.3} and \ref{16th4.4}. Theorem~\ref{16th4.3} deals with the local regularity of the generalized solution to the problem. Theorem~\ref{16th4.4} yields  sufficient conditions under which chosen generalized derivatives of the solution are continuous on a given set. These conditions are essentially finer than their versions obtained in the framework of the Sobolev spaces \cite{Ilin60, IlinKalashnikovOleinik62} and are sharp. Section~\ref{16sec6} is devoted to our basic research method, the interpolation with a function parameter between Hilbert spaces. The main results are proved in Section~\ref{16sec7}. Appendix discusses the equivalence of the compatibility conditions imposed on the right-hand sides of the parabolic problem to those considered in the cited  papers \cite{AgranovichVishik64, ZhitarashuEidelman98, Zhitarashu85}.

\section{Statement of the problem}\label{16sec2}

We arbitrarily choose an integer $n\geq2$ and a real number $\tau>0$. Suppose that $G$ is a bounded domain in $\mathbb{R}^{n}$ and that its boundary $\Gamma:=\partial G$ is an infinitely smooth closed manifold of dimension $n-1$. (Of course, the $C^{\infty}$-structure on $\Gamma$ is induced by~$\mathbb{R}^{n}$.) Put $\Omega:=G\times(0,\tau)$ and $S:=\Gamma\times(0,\tau)$; thus, $\Omega$ is an open cylinder in $\mathbb{R}^{n+1}$, and $S$ is its lateral area, with their closures $\overline{\Omega}=\overline{G}\times[0,\tau]$ and
$\overline{S}=\Gamma\times[0,\tau]$. We naturally identify $\overline{G}$ with the lower base $\{(x,0):x\in\overline{G}\}$ of the closed cylinder $\overline{\Omega}$.

We consider the following parabolic initial-boundary value problem in $\Omega$:
\begin{equation}\label{16f1}
\begin{gathered}
A(x,t,D_x,\partial_t)u(x,t) \equiv\sum_{|\alpha|+2b\beta\leq
2m}a^{\alpha,\beta}(x,t)\,D^\alpha_x\partial^\beta_t
u(x,t)=f(x,t)\\
\mbox{for all}\quad x\in G\quad\mbox{and}\quad t\in(0,\tau);
\end{gathered}
\end{equation}
\begin{equation}\label{16f2}
\begin{gathered}
B_{j}(x,t,D_x,\partial_t)u(x,t)\equiv\sum_{|\alpha|+2b\beta\leq m_j}
b_{j}^{\alpha,\beta}(x,t)\,D^\alpha_x\partial^\beta_t u(x,t)\!\upharpoonright\!S=g_{j}(x,t)\\
\mbox{for all}\quad x\in\Gamma,\quad t\in(0,\tau)\quad\mbox{and}\quad
j\in\{1,\dots,m\};
\end{gathered}
\end{equation}
\begin{equation}\label{16f3}
(\partial^k_t u)(x,0)=h_k(x)\quad\mbox{for all}\quad x\in G\quad\mbox{and}\quad
k\in\{0,\ldots,\varkappa-1\}.
\end{equation}
Here, $b$, $m$, and all $m_j$ are arbitrarily choosen integers that satisfy the conditions $m\geq b\geq1$, $\varkappa:=m/b\in\mathbb{Z}$, and $m_j\geq0$. All the coefficients of the linear partial differential expressions $A:=A(x,t,D_x,\partial_t)$ and $B_{j}:=B_{j}(x,t,D_x,\partial_t)$, with $j\in\{1,\dots,m\}$, are supposed to be infinitely smooth complex-valued functions given on $\overline{\Omega}$ and $\overline{S}$ respectively; i.e., each
\begin{equation*}
a^{\alpha,\beta}\in C^{\infty}(\overline{\Omega}):=
\bigl\{w\!\upharpoonright\overline{\Omega}\!:\,w\in C^{\infty}(\mathbb{R}^{n+1})\bigr\}
\end{equation*}
and each
\begin{equation*}
b_{j}^{\alpha,\beta}\in C^{\infty}(\overline{S}):=
\bigl\{v\!\upharpoonright\overline{S}\!:\,v\in C^{\infty}(\Gamma\times\mathbb{R})\bigr\}.
\end{equation*}

We use the notation
$D^\alpha_x:=D^{\alpha_1}_{1}\dots D^{\alpha_n}_{n}$, with $D_{k}:=i\,\partial/\partial{x_k}$, and $\partial_t:=\partial/\partial t$
for partial derivatives of functions depending on $x=(x_1,\ldots,x_n)\in\mathbb{R}^{n}$ and $t\in\mathbb{R}$. Here, $i$ is imaginary unit, and $\alpha=(\alpha_1,...,\alpha_n)$ is a multi-index, with $|\alpha|:=\alpha_1+\cdots+\alpha_n$. In formulas \eqref{16f1} and \eqref{16f2} and their analogs, we take summation over the integer-valued nonnegative indices $\alpha_1,...,\alpha_n$ and $\beta$ that satisfy the condition written under the integral sign. As usual, $\xi^{\alpha}:=\xi_{1}^{\alpha_{1}}\ldots\xi_{n}^{\alpha_{n}}$ for $\xi:=(\xi_{1},\ldots,\xi_{n})\in\mathbb{C}^{n}$.

We recall \cite[Section~9, Subsection~1]{AgranovichVishik64} that the initial-boundary value problem \eqref{16f1}--\eqref{16f3} is called parabolic in $\Omega$ if the following Conditions \ref{16cond2.1} and \ref{16cond2.2} are satisfied.

\begin{condition}\label{16cond2.1}
If $x\in\overline{G}$, $t\in[0,\tau]$,
$\xi\in\mathbb{R}^{n}$, and $p\in\mathbb{C}$ with $\mathrm{Re}\,p\geq0$, then
\begin{equation*}
A^{\circ}(x,t,\xi,p)\equiv\sum_{|\alpha|+2b\beta=2m} a^{\alpha,\beta}(x,t)\,\xi^\alpha
p^{\beta}\neq0\quad\mbox{whenever}\quad|\xi|+|p|\neq0.
\end{equation*}
\end{condition}

To formulate the next condition, we arbitrarily choose a point $x\in\Gamma$, real number $t\in[0,\tau]$, vector $\xi\in\mathbb{R}^{n}$ tangent to the boundary $\Gamma$ at $x$, and number $p\in\mathbb{C}$ with $\mathrm{Re}\,p\geq0$ such that $|\xi|+|p|\neq0$. Let $\nu(x)$ be the unit vector of the inward normal to $\Gamma$ at $x$. It follows from Condition~\ref{16cond2.1} and the inequality $n\geq2$ that the polynomial
$A^{\circ}(x,t,\xi+\zeta\nu(x),p)$ in $\zeta\in\mathbb{C}$ has $m$ roots $\zeta^{+}_{j}(x,t,\xi,p)$, $j=\nobreak1,\ldots,m$, with positive imaginary part and $m$ roots with negative imaginary part provided that each root is taken the number of times equal to its multiplicity.

\begin{condition}\label{16cond2.2}
The polynomials
$$
B_{j}^{\circ}(x,t,\xi+\zeta\nu(x),p)\equiv\sum_{|\alpha|+2b\beta=m_{j}}
b_{j}^{\alpha,\beta}(x,t)\,(\xi+\zeta\nu(x))^{\alpha}\,p^{\beta},\quad j=1,\dots,m,
$$
in $\zeta\in\mathbb{C}$ are linearly independent modulo
$$
\prod_{j=1}^{m}(\zeta-\zeta^{+}_{j}(x,t,\xi,p)).
$$
\end{condition}
Note Condition~\ref{16cond2.1} is that under which the partial differential equation $Au=\nobreak f$ is $2b$-parabolic in $\overline{\Omega}$ in the sense of I.~G.~Petrovskii \cite{Petrovskii38}. Besides, Condition~\ref{16cond2.2} means  that the system of boundary partial differential expressions $\{B_{1},\ldots,B_{m}\}$ covers $A$ on $\overline{S}$. This condition is introduced by Zagorskii \cite{Zagorskii61}, as noticed in \cite[\S~9, Subsection~1]{AgranovichVishik64}.

We investigate parabolic problem \eqref{16f1}--\eqref{16f3}
in appropriate generalized Sobolev spaces considered in the next section.

\section{Generalized Sobolev spaces related to the problem}\label{16sec3}

Throughout the paper, we use complex distribution spaces and interpret distributions as antilinear functionals. Among the normed distribution spaces $\mathcal{B}_{p,\mu}$ introduced and investigated by H\"ormander in \cite[Section~2.2]{Hermander63}, we need the inner product spaces $H^{\mu}(\mathbb{R}^{k}):=\mathcal{B}_{2,\mu}$, which give a broad generalization of the concept of Sobolev spaces (in the framework of Hilbert spaces). Here, $1\leq k\in\mathbb{Z}$, and $\nobreak{\mu:\mathbb{R}^{k}\rightarrow(0,\infty)}$ is an arbitrary Borel measurable function for which there exist positive numbers $c$ and $l$ such that
\begin{equation}\label{16f78}
\frac{\mu(\xi)}{\mu(\eta)}\leq
c\,(1+|\xi-\eta|)^{l}\quad\mbox{whenever}\quad \xi,\eta\in\mathbb{R}^{k}.
\end{equation}

By definition, the linear space $H^{\mu}(\mathbb{R}^{k})$ consists of all  distributions $w\in\mathcal{S}'(\mathbb{R}^{k})$ whose Fourier transform $\widehat{w}$ is a locally Lebesgue integrable function such that
\begin{equation*}
\int\limits_{\mathbb{R}^{k}}\mu^{2}(\xi)\,|\widehat{w}(\xi)|^{2}\,d\xi
<\infty.
\end{equation*}
The inner product in $H^{\mu}(\mathbb{R}^{k})$ is defined by the formula
\begin{equation*}
(w_1,w_2)_{H^{\mu}(\mathbb{R}^{k})}=
\int\limits_{\mathbb{R}^{k}}\mu^{2}(\xi)\,\widehat{w_1}(\xi)\,
\overline{\widehat{w_2}(\xi)}\,d\xi,
\end{equation*}
where $w_1,w_2\in H^{\mu}(\mathbb{R}^{k})$; this inner product induces
the norm
$$
\|w\|_{H^{\mu}(\mathbb{R}^{k})}:=(w,w)^{1/2}_
{H^{\mu}(\mathbb{R}^{k})}.
$$
As usual, $\mathcal{S}'(\mathbb{R}^{k})$ stands for the linear topological space of all tempered distributions on $\mathbb{R}^{k}$; this space is the antidual of the Schwartz space $\mathcal{S}(\mathbb{R}^{k})$ of rapidly decreasing functions on $\mathbb{R}^{k}$.

According to \cite[Section~2.2]{Hermander63}, the space $H^{\mu}(\mathbb{R}^{k})$ is Hilbert and separable with respect to this inner product. Besides, this space is continuously embedded in  $\mathcal{S}'(\mathbb{R}^{k})$, and the set
$\mathcal{S}(\mathbb{R}^{k})$ is dense in $H^{\mu}(\mathbb{R}^{k})$, as well as
the set $C^{\infty}_{0}(\mathbb{R}^{k})$ of all compactly supported  $C^\infty$-functions on $\mathbb{R}^{k}$ (see also \cite[Section~10.1]{Hermander83}). We will say
that the function parameter $\mu$ is the regularity index for the space $H^{\mu}(\mathbb{R}^{k})$ and its versions $H^{\mu}(\cdot)$.

A version of $H^{\mu}(\mathbb{R}^{k})$ for an
arbitrary nonempty open set $V\subset\mathbb{R}^{k}$ is introduced in the standard way. Namely,
\begin{gather}\notag
H^{\mu}(V):=\bigl\{w\!\upharpoonright\!V:\,
w\in H^{\mu}(\mathbb{R}^{k})\bigr\},\\
\|u\|_{H^{\mu}(V)}:= \inf\bigl\{\|w\|_{H^{\mu}(\mathbb{R}^{k})}:\,w\in
H^{\mu}(\mathbb{R}^{k}),\;u=w\!\upharpoonright\!V\bigr\}, \label{8f40}
\end{gather}
where $u\in H^{\mu}(V)$. Here, as usual, $w\!\upharpoonright\!V$ stands for the restriction of the distribution $w$ to the open set~$V$. In other words, $H^{\mu}(V)$ is the factor space of the separable Hilbert space $H^{\mu}(\mathbb{R}^{k})$ by its subspace
\begin{equation}\label{8f41}
H^{\mu}_{Q}(\mathbb{R}^{k}):=\bigl\{w\in
H^{\mu}(\mathbb{R}^{k}):\, \mathrm{supp}\,w\subseteq Q\bigr\}, \end{equation}
with $Q:=\mathbb{R}^{k}\backslash V$. Thus, $H^{\mu}(V)$ is also Hilbert and separable.
The norm \eqref{8f40} is induced by the inner product
$$
(u_{1},u_{2})_{H^{\mu}(V)}:= (w_{1}-\Upsilon
w_{1},w_{2}-\Upsilon w_{2})_{H^{\mu}(\mathbb{R}^{k})},
$$
where $w_{j}\in H^{\mu}(\mathbb{R}^{k})$, $w_{j}=u_{j}$ in $V$
for each $j\in\{1,\,2\}$, and $\Upsilon$ is the orthogonal projector of the space $H^{\mu}(\mathbb{R}^{k})$ onto its subspace \eqref{8f41}. The spaces $H^{\mu}(V)$ and $H^{\mu}_{Q}(\mathbb{R}^{k})$ were introduced and investigated by Volevich and Paneah
\cite[Section~3]{VolevichPaneah65}.

It follows directly from the definition of $H^{\mu}(V)$ and properties of $H^{\mu}(\mathbb{R}^{k})$ that the space $H^{\mu}(V)$ is continuously embedded in the linear topological space $\mathcal{D}'(V)$ of all distributions on $V$ and that the set $\{w\!\upharpoonright\!\overline{V}: w\in C^{\infty}_{0}(\mathbb{R}^{k})\}$ is dense in $H^{\mu}(V)$.

Suppose that the integer $k\geq2$, and arbitrarily choose a real number $\gamma>0$. We
need the H\"ormander spaces $H^{\mu}(\mathbb{R}^{k})$ and their versions in the case where the regularity index $\mu$ takes the form
\begin{equation}\label{16f70}
\begin{gathered}
\mu(\xi',\xi_{k})=\bigl(1+|\xi'|^2+|\xi_{k}|^{2\gamma}\bigr)^{s/2}
\varphi\bigl((1+|\xi'|^2+|\xi_{k}|^{2\gamma})^{1/2}\bigr)\\
\mbox{for all}\;\;\xi'\in\mathbb{R}^{k-1}\;\;\mbox{and}\;\;
\xi_{k}\in\mathbb{R}.
\end{gathered}
\end{equation}
Here, the number parameter $s$ is real, whereas the function parameter $\varphi$ runs over a certain class~$\mathcal{M}$.

By definition, the class $\mathcal{M}$ consists of all Borel measurable functions $\varphi:[1,\infty)\rightarrow(0,\infty)$ such that
\begin{itemize}
  \item [a)] both the functions $\varphi$ and $1/\varphi$ are bounded on each compact interval $[1,d]$, with $1<d<\infty$;
  \item [b)] the function $\varphi$ varies slowly at infinity in the sense of Karamata \cite{Karamata30a}, i.e.
      $\varphi(\lambda r)/\varphi(r)\rightarrow\nobreak1$ as $r\rightarrow\infty$ for every $\lambda>0$.
\end{itemize}

The theory of slowly varying functions (at infinity) is set forth in \cite{BinghamGoldieTeugels89, BuldyginIndlekoferKlesovSteinebach18, Seneta76}. Their standard  examples are the functions
\begin{equation*}
\varphi(r):=(\log r)^{\theta_{1}}\,(\log\log r)^{\theta_{2}} \ldots
(\,\underbrace{\log\ldots\log}_{k\;\mbox{\tiny{times}}}r\,)^{\theta_{k}}
\quad\mbox{of}\;\;r\gg1,
\end{equation*}
where the parameters $1\leq k\in\mathbb{Z}$ and
$\theta_{1},\theta_{2},\ldots,\theta_{k}\in\mathbb{R}$ are arbitrarily chosen.

Note that the regularity index \eqref{16f70} satisfies condition \eqref{16f78} (see \cite[Appendix]{LosMikhailetsMurach17CPAA}).
Dealing with the above-stated parabolic problem,
we need the H\"ormander spaces $H^{\mu}(\mathbb{R}^{k})$ with the regularity index \eqref{16f70}
only in the case where $\gamma=1/(2b)$. However, it is naturally to introduce
these spaces for arbitrary $\gamma>0$.

Let $s\in\mathbb{R}$ and $\varphi\in\mathcal{M}$.
We put $H^{s,s\gamma;\varphi}(\mathbb{R}^{k}):=H^{\mu}(\mathbb{R}^{k})$ in the case
where $\mu$ is of the form~\eqref{16f70}. Specifically, if $\varphi(r)\equiv1$,
then $H^{s,s\gamma;\varphi}(\mathbb{R}^{k})$ becomes the anisotropic Sobolev
inner product space $H^{s,s\gamma}(\mathbb{R}^{k})$ of order $(s,s\gamma)$
\cite{BesovIlinNikolskii75, Slobodetskii58}. Generally, if $\varphi\in\mathcal{M}$, we have the dense continuous embeddings:
\begin{equation}\label{8f5}
H^{s_{1},s_{1}\gamma}(\mathbb{R}^{k})\hookrightarrow
H^{s,s\gamma;\varphi}(\mathbb{R}^{k})\hookrightarrow
H^{s_{0},s_{0}\gamma}(\mathbb{R}^{k})\quad\mbox{whenever}\quad s_{0}<s<s_{1}.
\end{equation}
Indeed, let $s_{0}<s<s_{1}$; since $\varphi\in\mathcal{M}$, there exist positive numbers $c_0$ and $c_1$ such that $c_0\,r^{s_0-s}\leq\varphi(r)\leq c_1\,r^{s_1-s}$ for every $r\geq1$ (see e.g., \cite[Section 1.5, Property $1^\circ$]{Seneta76}).
Then
\begin{align*}
c_{0}\bigl(1+|\xi'|^2+|\xi_{k}|^{2\gamma}\bigr)^{s_{0}/2}&\leq
\bigl(1+|\xi'|^2+|\xi_{k}|^{2\gamma}\bigr)^{s/2}
\varphi\bigl((1+|\xi'|^2+|\xi_{k}|^{2\gamma})^{1/2}\bigr)\\
&\leq c_{1}\bigl(1+|\xi'|^2+|\xi_{k}|^{2\gamma}\bigr)^{s_{1}/2}
\end{align*}
for arbitrary $\xi'\in\mathbb{R}^{k-1}$ and $\xi_{k}\in\mathbb{R}$.
This directly entails the continuous embeddings~\eqref{8f5}.
They are dense because the set $C^{\infty}_{0}(\mathbb{R}^{k})$
is dense in all the spaces in \eqref{8f5}.

Consider the class of H\"ormander inner product spaces
\begin{equation}\label{8f6}
\bigl\{H^{s,s\gamma;\varphi}(\mathbb{R}^{k}):\,
s\in\mathbb{R},\,\varphi\in\mathcal{M}\,\bigr\}.
\end{equation}
The embeddings \eqref{8f5} show, that in \eqref{8f6}
the function parameter $\varphi$ defines subordinate regularity
with respect to the  basic anisotropic $(s,s\gamma)$-regularity.
Specifically, if $\varphi(r)\rightarrow\infty$
[or $\varphi(r)\rightarrow\nobreak0$] as $r\rightarrow\infty$,
then $\varphi$ defines supplementary positive [or negative] regularity.
In other words, $\varphi$ refines the basic regularity $(s,s\gamma)$.

We need versions of the function spaces \eqref{8f6} for the cylinder
$\Omega=G\times(0,\tau)$ and its lateral boundary  $S=\Gamma\times(0,\tau)$.
We put $H^{s,s\gamma;\varphi}(\Omega):=H^{\mu}(\Omega)$ in the case
where $\mu$ is of the form~\eqref{16f70} with $k:=n+1$.
For the function space $H^{s,s\gamma;\varphi}(\Omega)$,
the numbers $s$ and $s\gamma$ serve as the regularity indices of
distributions $u(x,t)$ with respect to the spatial variable $x\in G$
and to the time variable $t\in(0,\tau)$ respectively.

Following \cite[Section~1]{Los16JMathSci}, we will define the function space $H^{s,s\gamma;\varphi}(S)$ with the help of special local charts on~$S$. Let $s>0$ and $\varphi\in\mathcal{M}$. We put $H^{s,s\gamma;\varphi}(\Pi):=H^{\mu}(\Pi)$ for the strip $\Pi:=\mathbb{R}^{n-1}\times(0,\tau)$ in the case where $\mu$ is defined by formula \eqref{16f70} with $k:=n$. Recall that, according to our assumption, $\Gamma=\partial\Omega$ is an infinitely smooth closed manifold of dimension $n-1$, the $C^{\infty}$-structure on $\Gamma$ being induced by $\mathbb{R}^{n}$. From this structure we arbitrarily choose a finite atlas formed by local charts $\nobreak{\theta_{j}:\mathbb{R}^{n-1}\leftrightarrow \Gamma_{j}}$ with $j=1,\ldots,\lambda$. Here, the open sets $\Gamma_{1},\ldots,\Gamma_{\lambda}$
make up a covering of~$\Gamma$. We also arbitrarily choose  functions $\chi_{j}\in C^{\infty}(\Gamma)$, with $j=1,\ldots,\lambda$, such that $\mathrm{supp}\,\chi_{j}\subset\Gamma_{j}$ and $\chi_{1}+\cdots\chi_{\lambda}=1$ on $\Gamma$.

By definition, the linear space $H^{s,s\gamma;\varphi}(S)$ consists of all
square integrable functions $\nobreak{v:S\to\mathbb{C}}$ that the function
$$
v_{j}(y,t):=\chi_{j}(\theta_{j}(y))\,v(\theta_{j}(y),t)
\quad\mbox{of}\;\;y\in\mathbb{R}^{n-1}\;\;\mbox{and}\;\;t\in(0,\tau)
$$
belongs to $H^{s,s\gamma;\varphi}(\Pi)$ for each
$j\in\{1,\ldots,\lambda\}$. The inner product in $H^{s,s\gamma;\varphi}(S)$ is defined by the formula
\begin{equation*}
(v,v^\circ)_{H^{s,s\gamma;\varphi}(S)}:=\sum_{j=1}^{\lambda}\,
(v_{j},v^\circ_{j})_{H^{s,s\gamma;\varphi}(\Pi)},
\end{equation*}
where $v,v^\circ\in H^{s,s\gamma;\varphi}(S)$. This inner product induces the norm
$$
\|v\|_{H^{s,s\gamma;\varphi}(S)}:=(v,v)^{1/2}_{H^{s,s\gamma;\varphi}(S)}.
$$
The space $H^{s,s\gamma;\varphi}(S)$ is separable Hilbert one and does not depend up to equivalence of norms on the choice of local charts and partition of unity on $\Gamma$ \cite[Theorem~1]{Los16JMathSci}. (The proof in \cite{Los16JMathSci} is done
in the $\gamma\in\mathbb{Q}$ case we really need; the general case is treated similarly
to \cite[Lemma~3.1]{LosMikhailetsMurach17CPAA}).
Note that this space is actually defined with the help of the following special local charts on~$S$:
\begin{equation}\label{8f-local}
\theta_{j}^*:\Pi=\mathbb{R}^{n-1}\times(0,\tau)\leftrightarrow
\Gamma_{j}\times(0,\tau),\quad j=1,\ldots,\lambda,
\end{equation}
where $\theta_{j}^*(y,t):=(\theta_{j}(y),t)$ for all
$y\in\mathbb{R}^{n-1}$ and $t\in(0,\tau)$.

Let $s\in\mathbb{R}$ and $\varphi\in\mathcal{M}$. We also need the isotropic space $H^{s;\varphi}(V)$ over an arbitrary open nonempty set $V\subseteq\mathbb{R}^{k}$, with $k\geq1$. We put $H^{s;\varphi}(V):=H^{\mu}(V)$ in the case where
\begin{equation}\label{8f50}
\mu(\xi)=\bigl(1+|\xi|^2\bigr)^{s/2}\varphi\bigl((1+|\xi|^2)^{1/2}\bigr)
\quad\mbox{of}\;\;\xi\in\mathbb{R}^{k}.
\end{equation}
Since the function \eqref{8f50} is radial (i.e., depends only on
$|\xi|$), the space $H^{s;\varphi}(V)$ is isotropic.
We will use the spaces $H^{s;\varphi}(V)$ given over the whole Euclidean space $V:=\mathbb{R}^{k}$ or over the domain $V:=G$ in $\mathbb{R}^{n}$.

Besides, we will use the space $H^{s;\varphi}(\Gamma)$ over $\Gamma=\partial\Omega$. It is defined with the help of the above-mentioned collection of local charts $\{\theta_{j}\}$ and partition of unity $\{\chi_{j}\}$ on $\Gamma$ similarly to the spaces over $S$. By definition, the linear space  $H^{s;\varphi}(\Gamma)$ consists of all distributions $\omega$ on $\Gamma$ that for each number $j\in\{1,\ldots,\lambda\}$ the distribution
$\omega_{j}(y):=\chi_{j}(\theta_{j}(y))\,\omega(\theta_{j}(y))$ of
$y\in\mathbb{R}^{n-1}$ belongs to $H^{s;\varphi}(\mathbb{R}^{n-1})$.
The inner product in $H^{s;\varphi}(\Gamma)$ is defined by the formula
\begin{equation*}
(\omega,\omega^\circ)_{H^{s;\varphi}(\Gamma)}:=
\sum_{j=1}^{\lambda}\,
(\omega_{j},\omega^\circ_{j})_{H^{s;\varphi}(\mathbb{R}^{n-1})},
\end{equation*}
where $\omega,\omega^\circ\in H^{s;\varphi}(\Gamma)$. It induces the norm
$$
\|\omega\|_{H^{s;\varphi}(\Gamma)}:=
(\omega,\omega)^{1/2}_{H^{s;\varphi}(\Gamma)}.
$$
The space $H^{s;\varphi}(\Gamma)$ is separable Hilbert one and does not depend up to equivalence of norms on our choice of local charts and partition of unity on $\Gamma$ \cite[Theorem 2.1]{MikhailetsMurach14}.

Note that the classes of isotropic inner product spaces
\begin{equation*}
\bigl\{H^{s;\varphi}(V):s\in\mathbb{R},\;\varphi\in\mathcal{M}\bigr\}
\quad\mbox{and}\quad
\bigl\{H^{s;\varphi}(\Gamma):s\in\mathbb{R},\;\varphi\in\mathcal{M}\bigr\}
\end{equation*}
were selected, investigated, and systematically applied to elliptic differential operators and elliptic boundary-value problems by Mikhailets and Murach \cite{MikhailetsMurach14, MikhailetsMurach12BJMA2}.

If $\varphi\equiv1$, then the considered spaces $H^{s,s\gamma;\varphi}(\cdot)$ and $H^{s;\varphi}(\cdot)$ become the inner product Sobolev spaces $H^{s,s\gamma}(\cdot)$, anisotropic, and $H^{s}(\cdot)$, isotropic, respectively. It follows directly from \eqref{8f5} that
\begin{equation}\label{8f5a}
H^{s_{1},s_{1}\gamma}(\cdot)\hookrightarrow
H^{s,s\gamma;\varphi}(\cdot)\hookrightarrow
H^{s_{0},s_{0}\gamma}(\cdot)\quad\mbox{whenever}\quad s_{0}<s<s_{1}.
\end{equation}
Analogously,
\begin{equation}\label{8f5b}
H^{s_{1}}(\cdot)\hookrightarrow
H^{s;\varphi}(\cdot)\hookrightarrow
H^{s_{0}}(\cdot)\quad\mbox{whenever}\quad s_{0}<s<s_{1};
\end{equation}
see \cite[Theorems 2.3(iii) and 3.3(iii)]{MikhailetsMurach14}. These embeddings are continuous and dense. Certainly, if $s=0$, then $H^{s}(\cdot)=H^{s,s\gamma}(\cdot)$ is the Hilbert space $L_2(\cdot)$ of all square integrable functions given on the corresponding measurable set.

In the Sobolev case of $\varphi\equiv1$, we will omit the index $\varphi$ in designations of distribution spaces that will be introduced on the base of the spaces $H^{s,s\gamma;\varphi}(\cdot)$ and $H^{s;\varphi}(\cdot)$.


\section{Main results}\label{16sec4}

We  will formulate an isomorphism theorem for the parabolic problem \eqref{16f1}--\eqref{16f3} in the generalized Sobolev spaces introduced and then apply it to the investigation of regularity of solutions to the problem.

In order that a regular solution $u(x,t)$ to this problem exist, the
right-hand sides of the problem should satisfy certain compatibility conditions (see, e.g., \cite[Chapter~4, Section~5]{LadyzhenskajaSolonnikovUraltzeva67}).
These conditions consist in that the partial derivatives $(\partial^k_tu)(x,0)$, which could be found from the parabolic equation \eqref{16f1} and initial conditions \eqref{16f3}, should satisfy the boundary conditions \eqref{16f2} and some relations that are obtained by the differentiation of the boundary conditions with respect to~$t$. To write these compatibility conditions, we previously consider the problem in appropriate anisotropic Sobolev spaces.

We associate the linear mapping
\begin{equation}\label{16f4}
u\mapsto\Lambda
u:=\bigl(Au,B_1u,\ldots,B_mu,u\!\upharpoonright\!G,\ldots,
(\partial^{\varkappa-1}_tu)\!\upharpoonright\!G\bigr),\quad u\in
C^{\infty}(\overline{\Omega}),
\end{equation}
with the problem \eqref{16f1}--\eqref{16f3}.
Put
\begin{equation*}
\sigma_0:=\max\{2m,m_1+1,\dots,m_m+1\}.
\end{equation*}
(Specifficaly, if $m_j\leq2m-1$ for each $j\in\{1,\ldots,m\}$, then $\sigma_0=2m$.)
Let real $s\geq\sigma_0$; the mapping \eqref{16f4} extends uniquely (by continuity)
to a bounded linear operator
\begin{equation}\label{16f4a}
\Lambda:H^{s,s/(2b)}(\Omega)\rightarrow
\mathcal{H}^{s-2m,(s-2m)/(2b)},
\end{equation}
with
\begin{equation}\label{16rangeH}
\begin{aligned}
\mathcal{H}^{s-2m,(s-2m)/(2b)}
:= H^{s-2m,(s-2m)/(2b)}(\Omega)
&\oplus\bigoplus_{j=1}^{m}H^{s-m_j-1/2,(s-m_j-1/2)/(2b)}(S)\\
&\oplus\bigoplus_{k=0}^{\varkappa-1}H^{s-2bk-b}(G).
\end{aligned}
\end{equation}
This follows directly from the known properties of partial differential operators and trace operators on anisotropic Sobolev spaces (see, e.g., \cite[Chapter~I, Lemma~4, and Chapter~II, Theorems 3 and~7]{Slobodetskii58}). Choosing any function $u(x,t)$ from the space $H^{s,s/(2b)}(\Omega)$, we define the right-hand sides
\begin{equation}\label{16f4b}
\begin{gathered}
f\in H^{s-2m,(s-2m)/(2b)}(\Omega),\quad g_j\in H^{s-m_j-1/2,(s-m_j-1/2)/(2b)}(S),\quad
\mbox{and}\quad h_k\in H^{s-2bk-b}(G)\\
\mbox{for all}\quad j\in\{1,\dots,m\}\quad\mbox{and}\quad
k\in\{0,\dots,\varkappa-1\}
\end{gathered}
\end{equation}
of the problem by the formula
$$(f,g_1,...,g_m,h_0,...,h_{\varkappa-1}):=\Lambda u,$$
where $\Lambda$ is the operator \eqref{16f4a}.

The compatibility conditions for the functions $f$, $g_j$, and $h_k$ arise naturally in such a way. According to \cite[Chapter~II, Theorem 7]{Slobodetskii58}, the traces
$(\partial^{k}_t u)(\cdot,0)\in H^{s-2bk-b}(G)$ are well defined by closure for all $k\in\mathbb{Z}$ such that $0\leq k<s/(2b)-1/2$ (and only for these $k$).
These traces are expressed from \eqref{16f1} and \eqref{16f3} in terms of
$f$ and $h_k$ as follows.

The parabolicity Condition \ref{16cond2.1} in the case of $\xi=0$ and $p=1$ means that the coefficient $a^{(0,\ldots,0),\varkappa}(x,t)\neq0$ for all $x\in\overline{G}$ and $t\in[0,\tau]$.  We can therefore solve the parabolic equation \eqref{16f1} with respect to $\partial^\varkappa_t
u(x,t)$; namely,
\begin{equation}\label{16f5}
\partial^\varkappa_t u(x,t)=
\sum_{\substack{|\alpha|+2b\beta\leq 2m,\\ \beta\leq\varkappa-1}}
a_{0}^{\alpha,\beta}(x,t)\,D^\alpha_x\partial^\beta_t
u(x,t)+(a^{(0,\ldots,0),\varkappa}(x,t))^{-1}f(x,t),
\end{equation}
with $a_{0}^{\alpha,\beta}:=-a^{\alpha,\beta}/a^{(0,\ldots,0),\varkappa}\in C^{\infty}(\overline{\Omega})$. Let $k\in\mathbb{Z}$ satisfy $0\leq k<s/(2b)-1/2$. It follows from the initial conditions \eqref{16f3}, equality \eqref{16f5}, and the equalities obtained by the differentiation of \eqref{16f5} $k-\varkappa$ times with respect to $t$ that
\begin{equation}\label{16f6}
\begin{aligned}
(\partial^{k}_t u)(x,0)&=h_k(x)\quad\mbox{if}
\quad 0\leq k\leq\varkappa-1,\\
(\partial^{k}_t u)(x,0)&=
\sum_{\substack{|\alpha|+2b\beta\leq 2m,\\ \beta\leq\varkappa-1}}\,
\sum\limits_{q=0}^{k-\varkappa}
\binom{k-\varkappa}{q}(\partial^{k-\varkappa-q}_t
a_{0}^{\alpha,\beta})(x,0)\,D^\alpha_x(\partial^{\beta+q}_tu)(x,0)+\\
&+\partial^{k-\varkappa}_t((a^{(0,\ldots,0),\varkappa})^{-1}
f)(x,0)\quad\mbox{if}\quad k\geq\varkappa.
\end{aligned}
\end{equation}
These equalities hold for almost all $x\in G$, and partial derivatives are interpreted in the sense of the theory of distributions.

Besides, according to \cite[Chapter~II, Theorem 7]{Slobodetskii58}, for each $j\in\{1,\dots,m\}$
the traces $\partial^{\,k}_t g_j(\cdot,0)\in H^{s-m_j-1/2-2bk-b}(\Gamma)$
are well defined by closure for all $k\in\mathbb{Z}$
such that $0\leq k<(s-m_j-1/2-b)/(2b)$ (and only for these $k$).
We can express these traces in terms of the function $u(x,t)$ and its time derivatives; namely,
\begin{equation}\label{16f4bb}
\begin{aligned}
(\partial^{k}_t g_j)(x,0)&=(\partial^{k}_{t}B_ju)(x,0)\\
&=\sum_{|\alpha|+2b\beta\leq m_j}\,
\sum_{q=0}^{k}\binom{k}{q}
(\partial^{k-q}_t b^{\alpha,\beta}_j)(x,0)\,
D^\alpha_x(\partial^{\beta+q}_t u)(x,0)
\end{aligned}
\end{equation}
for almost all $x\in\Gamma$. Here, all the functions
$$
u(x,0),(\partial_{t}u)(x,0),\ldots,(\partial^{\,[m_j/(2b)]+k}_{t}u)(x,0)
$$
of $x\in G$ are expressed in terms of the functions $f(x,t)$ and $h_{k}(x)$ by the recurrent formula \eqref{16f6}. (As usual, $[m_j/(2b)]$ denotes the integral part of $m_j/(2b)$.)

Substituting \eqref{16f6} in the right-hand side of \eqref{16f4bb}, we obtain the compatibility conditions
\begin{equation}\label{16f8}
\begin{gathered}
\partial^{k}_t g_j\!\upharpoonright\!\Gamma=
B_{j,k}(v_0,\dots,v_{[m_j/(2b)]+k})\!\upharpoonright\!\Gamma \\
\mbox{for each}\;\;j\in\{1,\dots,m\}\;\;\mbox{and}\;\;k\in\mathbb{Z}
\;\;\mbox{such that}\;\;0\leq k<\frac{s-m_j-1/2-b}{2b}.
\end{gathered}
\end{equation}
Here,
the functions $v_0, v_1,\ldots$ are defined almost everywhere on $G$  by the formulas
\begin{equation}\label{16f9}
\begin{aligned}
v_k(x)&=h_k(x)\quad\mbox{if}
\quad 0\leq k\leq\varkappa-1,\\
v_k(x)&=
\sum_{\substack{|\alpha|+2b\beta\leq 2m,\\ \beta\leq\varkappa-1}}\,
\sum\limits_{q=0}^{k-\varkappa}
\binom{k-\varkappa}{q}(\partial^{\,k-\varkappa-q}_t
a_{0}^{\alpha,\beta})(x,0)\,D^\alpha_x v_{\beta+q}(x)+\\
&\qquad\quad+\partial^{k-\varkappa}_t ((a^{(0,\ldots,0),\varkappa})^{-1}f)(x,0)
\quad\mbox{if}\quad k\geq\varkappa,
\end{aligned}
\end{equation}
and
\begin{equation}\label{16f9B}
B_{j,k}(v_0,\dots,v_{[m_j/(2b)]+k})(x)=\sum_{|\alpha|+2b\beta\leq m_j}
\sum_{q=0}^{k}\binom{k}{q}(\partial^{k-q}_t b^{\alpha,\beta}_j)(x,0)\,
D^\alpha_x v_{\beta+q}(x)
\end{equation}
for almost all $x\in G$. Note, that
\begin{equation*}
v_k\in H^{s-2bk-b}(G)\quad\mbox{for each}\quad
k\in\mathbb{Z}\cap[0,s/(2b)-1/2)
\end{equation*}
due to \eqref{16f4b}.
The right-hand side of the equality in \eqref{16f8} is well defined because the function
$B_{j,k}(v_0,\dots,v_{[m_j/(2b)]+k})$ belongs to $H^{s-m_j-2bk-b}(G)$ and the trace
\begin{equation}\label{8f69aa}
B_{j,k}(v_0,\dots,v_{[m_j/(2b)]+k})\!\upharpoonright\!\Gamma\in H^{s-m_j-2bk-b-1/2}(\Gamma)
\end{equation}
is therefore defined by closure whenever $s-m_j-2bk-b-1/2>0$.

The number of the compatibility conditions \eqref{16f8} is a function of
$s\geq\sigma_0$. This function is discontinuous at $s$ if and only if
$(s-m_j-1/2-b)/(2b)\in\mathbb{Z}$. Thus, the set of all its discontinuities
coincides with
\begin{equation}\label{setE}
E:=\{(2l+1)b+m_j+1/2:j,l\in\mathbb{Z},\;1\leq j\leq m,\;l\geq0\}
\cap(\sigma_0,\infty).
\end{equation}
Note, if $s\leq\min\{m_1,\ldots,m_m\}+b+1/2$, there are no compatibility conditions.

Our main result on the parabolic problem \eqref{16f1}--\eqref{16f3} consists in that the linear mapping \eqref{16f4} extends uniquely  to an isomorphism between appropriate pairs of generalized Sobolev spaces introduced in the previous section. Let us indicate these pairs. We arbitrarily choose a real number $s>\sigma_0$ and function parameter $\varphi\in\mathcal{M}$. We also consider the Sobolev case where $s=\sigma_0$ and $\varphi\equiv1$ (we need it to formulate our results).
We take $H^{s,s/(2b);\varphi}(\Omega)$ as the domain of this isomorphism. Its range is imbedded in the Hilbert space
\begin{align*}
\mathcal{H}^{s-2m,(s-2m)/(2b);\varphi}:=
H^{s-2m,(s-2m)/(2b);\varphi}(\Omega)
&\oplus\bigoplus_{j=1}^{m}H^{s-m_j-1/2,(s-m_j-1/2)/(2b);\varphi}(S)\\
&\oplus\bigoplus_{k=0}^{\varkappa-1}H^{s-2bk-b;\varphi}(G)
\end{align*}
and is denoted by  $\mathcal{Q}^{s-2m,(s-2m)/(2b);\varphi}$.
[If $\varphi\equiv1$, then $\mathcal{H}^{s-2m,(s-2m)/(2b);\varphi}$ is the target space of \eqref{16f4a}.]
We separately define $\mathcal{Q}^{s-2m,(s-2m)/(2b);\varphi}$
in the cases where $s\notin E$ and where $s\in E$.

Suppose first that $s\notin E$. By definition, the linear space $\mathcal{Q}^{s-2m,(s-2m)/(2b);\varphi}$ consists of all vectors
$$
F:=\bigl(f,g_1,\dots,g_m,h_0,\dots,h_{\varkappa-1}\bigr)\in
\mathcal{H}^{s-2m,(s-2m)/(2b);\varphi}
$$
that satisfy the compatibility conditions \eqref{16f8}.
These conditions are well defined for every indicated $F$ because
they are well defined whenever
$F\in\mathcal{H}^{s-\varepsilon-2m,(s-\varepsilon-2m)/(2b)}$ and
$0<\varepsilon\ll 1$ and because
\begin{equation}\label{8f69a}
\mathcal{H}^{s-2m,(s-2m)/(2b);\varphi}\hookrightarrow
\mathcal{H}^{s-\varepsilon-2m,(s-\varepsilon-2m)/(2b)}.
\end{equation}
[This continuous embedding  follows directly from \eqref{8f5a} and \eqref{8f5b}.]
We endow the linear space $\mathcal{Q}^{s-2m,(s-2m)/(2b);\varphi}$ with the inner product and norm in the Hilbert space
$\mathcal{H}^{s-2m,(s-2m)/(2b);\varphi}$. The space $\mathcal{Q}^{s-2m,(s-2m)/(2b);\varphi}$
is complete, i.e. Hilbert. Indeed,
$$
\mathcal{Q}^{s-2m,(s-2m)/(2b);\varphi}=
\mathcal{H}^{s-2m,(s-2m)/(2b);\varphi}\cap
\mathcal{Q}^{s-\varepsilon-2m,(s-\varepsilon-2m)/(2b)}
$$
whenever $0<\varepsilon\ll 1$.
Here, the space $\mathcal{Q}^{s-\varepsilon-2m,(s-\varepsilon-2m)/(2b)}$ is complete because the differential operators and trace operators used in the compatibility conditions are bounded on the corresponding pairs of Sobolev spaces. Therefore, the right-hand side of this equality is complete with respect to the sum of the norms in the components of the intersection, this sum being equivalent to the norm in $\mathcal{H}^{s-2m,(s-2m)/(2b);\varphi}$ due to \eqref{8f69a}. Thus, the space $\mathcal{Q}^{s-2m,(s-2m)/(2b);\varphi}$ is complete (with respect to the latter norm).

If $s\in E$, then we define the Hilbert space $\mathcal{Q}^{s-2m,(s-2m)/(2b);\varphi}$
by means of the interpolation between its analogs just introduced. Namely, we put
\begin{equation}\label{16f10}
\mathcal{Q}^{s-2m,(s-2m)/(2b);\varphi}:=
\bigl[\mathcal{Q}^{s-\varepsilon-2m,(s-\varepsilon-2m)/(2b);\varphi},
\mathcal{Q}^{s+\varepsilon-2m,(s+\varepsilon-2m)/(2b);\varphi}\bigr]_{1/2}.
\end{equation}
Here, the number $\varepsilon\in(0,1/2)$ is arbitrarily chosen, and the right-hand side of the equality is the result of the interpolation with the parameter~$1/2$
of the written pair of Hilbert spaces. We will recall the definition of the interpolation between Hilbert spaces in  Section~\ref{16sec6}. The Hilbert space $\mathcal{Q}^{s-2m,(s-2m)/(2b);\varphi}$ defined by formula \eqref{16f10} does not depend on our choice of $\varepsilon$ up to equivalence of norms and is continuously embedded in $\mathcal{H}^{s-2m,(s-2m)/(2b);\varphi}$. This will be shown in Remark~\ref{16rem8.1}.

\begin{theorem}\label{16th4.1}
For arbitrary $s>\sigma_0$ and $\varphi\in\nobreak\mathcal{M}$, the mapping \eqref{16f4} extends
uniquely (by continuity) to an isomorphism
\begin{equation}\label{16f11}
\Lambda :\,H^{s,s/(2b);\varphi}(\Omega)\leftrightarrow
\mathcal{Q}^{s-2m,(s-2m)/(2b);\varphi}.
\end{equation}
\end{theorem}

This theorem is known in the Sobolev case where $\varphi\equiv1$.
It is proved in this case by Agranovich and Vishik
\cite[Theorem~12.1]{AgranovichVishik64} under the restriction  $s,s/(2b)\in\mathbb{Z}$. This restriction can be removed; see, e.g.,
Lions and Magenes' monograph \cite[Theorem~6.2]{LionsMagenes72ii}
in the case of $b=1$ and the normal boundary conditions, and Zhitarashu's paper \cite[Theorem~9.1]{Zhitarashu85} in the general case.
Their results include the limiting case of $s=\sigma_0$.
Note that these papers deal with another equivalent form of the compatibility conditions \eqref{16f8}, which will be discussed in Appendix.

We will deduce Theorem~\ref{16th4.1} from the Sobolev case with the help of the interpolation with a function parameter between Hilbert spaces.
This will be done in Section~\ref{16sec7} after we investigate the necessary interpolation properties of the spaces used in \eqref{16f11}.

Note that we have to define the range of the isomorphism \eqref{16f11} by the interpolation formula \eqref{16f10} in the  $s\in E$ case because
this isomorphism can cease holding if we define
$\mathcal{Q}^{s-2m,(s-2m)/(2b);\varphi}$ in the way used in the $s\notin E$ case.
This is suggested by Solonnikov's result \cite[Section~6]{Solonnikov64} concerning the heat equation in Sobolev spaces; see also \cite[Remark~6.4]{LionsMagenes72ii}.

Let us discuss the regularity properties of the generalized solution to the parabolic problem \eqref{16f1}--\eqref{16f3}. We assume further in this section that the right-hand sides $f$, $g_{j}$, and $h_k$ of the problem are arbitrary distributions given respectively on $\Omega$, $S$ and $G$. A function
$u\in H^{\sigma_0,\sigma_0/(2b)}(\Omega)$ is said to be a (strong) generalized solution to this problem if
\begin{equation*}
\Lambda u=(f,g_{1},\dots,g_{m},h_0,\dots,h_{\varkappa-1});
\end{equation*}
here, $\Lambda$ is the bounded operator \eqref{16f4a} for $s:=\sigma_0$. It follows from this condition that
\begin{equation}\label{16f12}
(f,g_{1},\dots,g_{m},h_0,\dots,h_{\varkappa-1})\in
\mathcal{Q}^{\sigma_0-2m,(\sigma_0-2m)/(2b)}.
\end{equation}
Moreover \cite[Theorem~9.1]{Zhitarashu85}, the problem has a unique solution $u\in H^{\sigma_0,\sigma_0/(2b)}(\Omega)$ for every vector \eqref{16f12}. We see now that the following result is a direct consequence of Theorem~\ref{16th4.1}:

\begin{corollary}\label{16cor4.2}
Assume that a function $u\in H^{\sigma_0,\sigma_0/(2b)}(\Omega)$
is a generalized solution to the parabolic problem
\eqref{16f1}--\eqref{16f3} whose right-hand sides satisfy the condition
$$
(f,g_{1},\dots,g_{m},h_0,\dots,h_{\varkappa-1})\in
\mathcal{Q}^{s-2m,(s-2m)/(2b);\varphi}
$$
for some $s>\sigma_0$ and $\varphi\in\mathcal{M}$.
Then $u\in H^{s,s/(2b);\varphi}(\Omega)$.
\end{corollary}

Let us formulate a local version of this result.
Let $U$ be an open subset of $\mathbb{R}^{n+1}$ such that $\Omega_0:=U\cap\Omega\neq\emptyset$ and $U\cap\Gamma=\emptyset$.
Put $\Omega':=U\cap\partial\overline{\Omega}$, $S_0:=U\cap S$, $S':=U\cap \{(x,\tau);x\in\Gamma\}$, and $G_0:=U\cap G$.
We need to introduce local versions of the spaces
$H^{s,s/(2b);\varphi}(\Omega)$,  $H^{s,s/(2b);\varphi}(S)$ and $H^{s;\varphi}(G)$
with $s>0$ and $\varphi\in\mathcal{M}$.

We let $H^{s,s/(2b);\varphi}_{\mathrm{loc}}(\Omega_0,\Omega')$ denote the linear space of all distributions $u$ in $\Omega$
such that $\chi u\in H^{s,s/(2b);\varphi}(\Omega)$ for every function $\chi\in C^\infty (\overline\Omega)$ subject to $\mathrm{supp}\,\chi\subset\Omega_0\cup\Omega'$.
Analogously, $H^{s,s/(2b);\varphi}_{\mathrm{loc}}(S_0,S')$ denotes the linear space of all distributions $v$ on $S$ such that $\chi v\in H^{s,s/(2b);\varphi}(S)$ for every function $\chi\in C^\infty (\overline S)$ subject to $\mathrm{supp}\,\chi\subset S_0\cup S'$. Finally, $H^{s;\varphi}_{\mathrm{loc}}( G_0)$ stands for the linear space of all distributions $w$ in $G$
such that $\chi w\in H^{s;\varphi}(G)$ for every function $\chi\in C^\infty (\overline G)$ satisfying $\mathrm{supp}\,\chi\subset G_0$.

\begin{theorem}\label{16th4.3}
Let $s>\sigma_0$ and $\varphi\in\mathcal{M}$. Assume that a function
$u\in H^{\sigma_0,\sigma_0/(2b)}(\Omega)$ is a generalized solution to the parabolic problem \eqref{16f1}--\eqref{16f3} whose right-hand sides satisfy the following conditions:
\begin{align}\label{16f13}
f&\in H^{s-2m,(s-2m)/(2b);\varphi}_{\mathrm{loc}}
(\Omega_0,\Omega'),
\\ \label{16f14}
g_{j}&\in H^{s-m_j-1/2,(s-m_j-1/2)/(2b);\varphi}_{\mathrm{loc}}
(S_0,S')\quad\mbox{for each}\quad j\in\{1,\dots,m\},
\\ \label{16f15}
h_{k}&\in H^{s-2bk-b;\varphi}_{\mathrm{loc}}
( G_0)\quad\mbox{for each}\quad k\in\{0,\dots,\varkappa-1\}.
\end{align}
Then $u\in H^{s,s/(2b);\varphi}_{\mathrm{loc}}(\Omega_0,\Omega')$.
\end{theorem}

If $\Omega'=\emptyset$, Theorem~\ref{16th4.3} asserts that
the regularity of $u$ increases on neighbourhoods of internal points
of $\overline{\Omega}$. If $G_0=\emptyset$, this theorem states
that the regularity of $u(x,t)$ increases whenever $t>0$. In this
case, the theorem follows directly from \cite[Theorem~4.3]{LosMikhailetsMurach17CPAA} provided $\sigma_0/(2b)\in\mathbb{Z}$.
Remark that we restrict ourselves to the case $U\cap\Gamma=\emptyset$ because
the conclusion of Theorem~\ref{16th4.3} is not true in the general case.

Using the spaces introduced, we obtain sufficient conditions under which the generalized solution $u$ and its generalized derivatives of a prescribed order are continuous on $\Omega_0\cup\Omega'$.

\begin{theorem}\label{16th4.4}
Let an integer $p\geq0$ satisfy $p+b+n/2>\sigma_{0}$. Assume that a function $u\in H^{\sigma_0,\sigma_0/(2b)}(\Omega)$ is a generalized solution to the parabolic problem \eqref{16f1}--\eqref{16f3} whose right-hand sides satisfy conditions \eqref{16f13}--\eqref{16f15} for $s:=p+b+n/2$ and some function parameter $\varphi\in\mathcal{M}$ subject to
\begin{equation}\label{9f4.7}
\int\limits_{1}^{\infty}\frac{dr}{r\varphi^2(r)}<\infty.
\end{equation}
Then the solution $u(x,t)$ and all its generalized derivatives
$D_{x}^{\alpha}\partial_{t}^{\beta}u(x,t)$ with $|\alpha|+2b\beta\leq p$ are continuous on $\Omega_0\cup\Omega'$.
\end{theorem}

As to the conclusion of this theorem, note that a distribution $v$ in $\Omega$ is called continuous on the set $\Omega_0\cup\Omega'$ if there exists a continuous function $v_{0}$ on $\Omega_0\cup\Omega'$ such that
\begin{equation}\label{rem-to-th4.4}
v(\omega)=\int\limits_{\Omega_0}v_{0}(x,t)\,\omega(x,t)\,dxdt
\end{equation}
for every test function $\omega\in C^{\infty}(\Omega)$ subject to $\mathrm{supp}\,\omega\subset\Omega_0$. Here, $v(\omega)$ stands for the value of the functional $v$ at $\omega$. (It is not difficult to show that this definition is equivalent to the following: $\chi v\in C(\overline\Omega)$ for every function $\chi\in C^\infty (\overline\Omega)$ such that $\mathrm{supp}\,\chi\subset\Omega_0\cup\Omega'$.)

\begin{remark}\label{16rem4.5}\rm
Condition \eqref{9f4.7} in Theorem~\ref{16th4.4} is sharp. Namely, let $s:=p+b+n/2$ and $\varphi\in\mathcal{M}$, and assume that for every function $u\in\nobreak H^{\sigma_0,\sigma_0/(2b)}(\Omega)$ the following implication holds:
\begin{equation}\label{f-rem4.5}
\begin{aligned}
&\bigl(u\;\mbox{is a solution to problem \eqref{16f1}--\eqref{16f3} for some right-hand sides \eqref{16f13}--\eqref{16f15}}\bigr)\\
&\Longrightarrow\bigl(u\;\mbox{satisfies the conclusion of Theorem \ref{16th4.4}}\bigr).
\end{aligned}
\end{equation}
Then $\varphi$ satisfies condition \eqref{9f4.7}.
\end{remark}

Note also that the use of generalized Sobolev spaces allows us to obtain a finer result then it is possible in the framework of Sobolev spaces. Namely, if we formulate an analog of Theorem~\ref{16th4.4} for the Sobolev case of $\varphi\equiv\nobreak1$, we have to replace the condition of this theorem with a stronger one.  Thus, we have to claim that the right-hand sides of the problem \eqref{16f1}--\eqref{16f3} satisfy conditions \eqref{16f13}--\eqref{16f15} for certain $s>p+b+n/2$. This claim is stronger than the condition of Theorem~\ref{16th4.4} due to the left-hand embeddings in~\eqref{8f5a} and \eqref{8f5b}. This theorem can be used to obtain sufficient conditions under which the generalized solution $u$ to the parabolic problem is classical (see \cite{Los17UMG11}).

We will prove Theorems~\ref{16th4.3} and \ref{16th4.4} at the end of
Section~\ref{16sec7} and then substantiate Remark~\ref{16rem4.5}.

\section{Interpolation with a function parameter between Hilbert spaces}\label{16sec6}

This method of interpolation is a natural generalization of the classical interpolation method by S.~Krein and J.-L.~Lions (see their monographs \cite[Chapter~IV, Section~1, Subsection~10]{KreinPetuninSemenov82} and \cite[Chapter~1, Sections 2 and 5]{LionsMagenes72i}) to the case where a general enough function is used instead of a number as an interpolation parameter. We restrict ourselves to the case of separable complex Hilbert spaces and mainly follow the monograph \cite[Section~1.1]{MikhailetsMurach14}.

Let $X:=[X_{0},X_{1}]$ be an ordered pair of separable complex Hilbert spaces such that $X_1$ is a dense linear manifold in $X_0$ and that the
embedding $X_{1}\subseteq X_{0}$ is continuous.
This pair is called regular. For $X$ there is a positive-definite self-adjoint operator $J$ in $X_{0}$ with the domain $X_{1}$ such that $\|Jv\|_{X_{0}}=\|v\|_{X_{1}}$ for every $v\in X_{1}$. This operator is uniquely determined by $X$ and is called the generating operator for~$X$; see, e.g., \cite[Chapter~IV, Theorem~1.12]{KreinPetuninSemenov82}. The operator sets an isometric isomorphism between $X_{1}$ and $X_{0}$.

Let $\mathcal{B}$ denote the set of all Borel measurable functions $\psi:(0,\infty)\rightarrow(0,\infty)$ such that $\psi$ is bounded on each compact interval $[a,b]$, with $0<a<b<\infty$, and that $1/\psi$ is bounded on every semiaxis $[a,\infty)$, with $a>0$.

Choosing a function $\psi\in\mathcal{B}$ arbitrarily, we consider the (generally, unbounded) operator $\psi(J)$ in $X_{0}$ as the Borel function $\psi$ of $J$. This operator is built with the help of Spectral Theorem applied to the self-adjoint operator $J$. Let $[X_{0},X_{1}]_{\psi}$ or, simply, $X_{\psi}$ denote the domain of $\psi(J)$ endowed with the inner product $(v_{1},v_{2})_{X_{\psi}}:=(\psi(J)v_{1},\psi(J)v_{2})_{X_{0}}$
and the corresponding norm $\|v\|_{X_{\psi}}:=\|\psi(J)v\|_{X_{0}}$. The linear space $X_{\psi}$ is Hilbert and separable with respect to this norm.

A function $\psi\in\mathcal{B}$ is called an interpolation parameter if the following condition is satisfied for all regular pairs $X=[X_{0},X_{1}]$ and $Y=[Y_{0},Y_{1}]$ of Hilbert spaces and for an arbitrary linear mapping $T$ given on $X_{0}$: if the restriction of $T$ to $X_{j}$ is a bounded operator $T:X_{j}\rightarrow Y_{j}$ for each $j\in\{0,1\}$, then the restriction of $T$ to
$X_{\psi}$ is also a bounded operator $T:X_{\psi}\rightarrow Y_{\psi}$.

If $\psi$ is an interpolation parameter, we will say that the Hilbert space $X_{\psi}$ is obtained by the interpolation with the function parameter $\psi$ of the pair $X=\nobreak[X_{0},X_{1}]$ or, otherwise speaking, between the spaces $X_{0}$ and $X_{1}$. In this case, the dense and continuous embeddings $X_{1}\hookrightarrow
X_{\psi}\hookrightarrow X_{0}$ hold.

The class of all interpolation parameters (in the sense of the given definition) admits a constructive description. Namely, a function $\psi\in\mathcal{B}$ is an interpolation parameter if and only if $\psi$ is pseudoconcave in a neighbourhood of infinity. The latter property means that there exists a concave positive function $\psi_{1}(r)$ of $r\gg1$ that both the functions $\psi/\psi_{1}$ and $\psi_{1}/\psi$ are bounded in some neighbourhood of infinity. This criterion follows from Peetre's description of all interpolation functions for the weighted Lebesgue spaces \cite{Peetre66, Peetre68} (this result of Peetre is set forth in the monograph \cite[Theorem 5.4.4]{BerghLefstrem76}). The proof of the criterion is given, e.g., in \cite[Section 1.1.9]{MikhailetsMurach14}.

The application of this criterion to power functions gives the classical result by Krein and Lions. Namely, the function $\psi(r)\equiv r^{\theta}$ is an interpolation parameter if and only if $\nobreak{0\leq\theta\leq1}$. In this case, the exponent $\theta$ serves as a number parameter of the interpolation, and the interpolation space $X_{\psi}$ is also denoted by $X_{\theta}$. We used this interpolation in our definition \eqref{16f10}, with $\theta=1/2$.

For the readers' convenience, we formulate the general interpolation properties used systematically below. The first of them enables us to reduce the interpolation of subspaces to the interpolation of the whole spaces (see \cite[Theorem~1.6]{MikhailetsMurach14} or \cite[Section~1.17.1, Theorem~1]{Triebel95}). As usual, subspaces of normed spaces are supposed to be closed. Generally, we consider nonorthogonal projectors onto subspaces of a Hilbert space.

\begin{proposition}\label{8prop1}
Let $X=[X_{0},X_{1}]$ be a regular pair of Hilbert spaces, and let $Y_{0}$ be a subspace of $X_{0}$. Then $Y_{1}:=X_{1}\cap Y_{0}$ is a subspace of $X_{1}$. Suppose that there exists a linear mapping $P$ on $X_{0}$ such that $P$ is a projector of the space $X_{j}$ onto its subspace $Y_{j}$ for each $j\in\{0,\,1\}$. Then the pair $[Y_{0},Y_{1}]$ is regular, and $[Y_{0},Y_{1}]_{\psi}=X_{\psi}\cap Y_{0}$ with equivalence of norms for an arbitrary interpolation parameter~$\psi\in\mathcal{B}$. Here, $X_{\psi}\cap Y_{0}$ is a subspace of $X_{\psi}$.
\end{proposition}

The second property reduces the interpolation of orthogonal sums of Hilbert spaces to the interpolation of their summands (see \cite[Theorem~1.8]{MikhailetsMurach14}).

\begin{proposition}\label{8prop2}
Let $[X_{0}^{(j)},X_{1}^{(j)}]$, with $j=1,\ldots,q$, be a finite collection of regular pairs of Hilbert spaces. Then
$$
\biggl[\,\bigoplus_{j=1}^{q}X_{0}^{(j)},\,
\bigoplus_{j=1}^{q}X_{1}^{(j)}\biggr]_{\psi}=\,
\bigoplus_{j=1}^{q}\bigl[X_{0}^{(j)},\,X_{1}^{(j)}\bigr]_{\psi}
$$
with equality of norms for every function $\psi\in\mathcal{B}$.
\end{proposition}

The third property  is  Reiteration Theorem for the interpolation \cite[Theorem~1.3]{MikhailetsMurach14}.

\begin{proposition}\label{8prop3}
Let $\alpha,\beta,\psi\in\mathcal{B}$, and suppose that the function $\alpha/\beta$ is bounded in a neighbourhood of infinity. Define the function $\omega\in\mathcal{B}$  by the formula
$\omega(r):=\alpha(r)\psi(\beta(r)/\alpha(r))$ for $r>0$. Then $\omega\in\mathcal{B}$, and $[X_{\alpha},X_{\beta}]_{\psi}=X_{\omega}$ with equality of norms for every regular pair $X$ of Hilbert spaces. Besides, if $\alpha,\beta,\psi$ are interpolation parameters, then
$\omega$ is also an interpolation parameter.
\end{proposition}

Our proof of Theorem \ref{16th4.1} is based on the key fact that the interpolation with an appropriate function parameter between marginal Sobolev spaces in \eqref{8f5a} and \eqref{8f5b} gives the intermediate spaces $H^{s,s\gamma;\varphi}(\cdot)$ and $H^{s;\varphi}(\cdot)$ respectively. Let us formulate this property separately for isotropic and for anisotropic spaces.

\begin{proposition}\label{8prop4}
Let $s_{0},s,s_{1}\in\mathbb{R}$  satisfy
$s_{0}<s<s_{1}$, and let $\varphi\in\mathcal{M}$. Put
\begin{equation}\label{8f16}
\psi(r):=
\begin{cases}
\;r^{(s-s_{0})/(s_{1}-s_{0})}\,\varphi(r^{1/(s_{1}-s_{0})})&\text{if}
\quad r\geq1,\\
\;\varphi(1) & \text{if}\quad0<r<1.
\end{cases}
\end{equation}
Then the function $\psi\in\mathcal{B}$ is an interpolation parameter, and the equality of spaces
\begin{equation}\label{8f49}
H^{s-\lambda;\varphi}(W)=
\bigl[H^{s_{0}-\lambda}(W),H^{s_{1}-\lambda}(W)\bigr]_{\psi}
\end{equation}
holds true with equivalence of norms for arbitrary $\lambda\in\mathbb{R}$ provided that $W=G$ or $W=\Gamma$. If $W=\mathbb{R}^{k}$ with $1\leq k\in\mathbb{Z}$, then \eqref{8f49} holds true with equality of norms.
\end{proposition}

The proof of this proposition is given in \cite[Theorems 1.14, 2.2, and 3.2]{MikhailetsMurach14} for the cases where $W=\mathbb{R}^{k}$, $W=\Gamma$, and $W=G$ respectively.

\begin{proposition}\label{8prop5}
Let $s_{0},s,s_{1}\in\mathbb{R}$
satisfy $0\leq s_{0}<s<s_{1}$, and let $\varphi\in\mathcal{M}$. Define the interpolation parameter $\psi\in\mathcal{B}$ by formula \eqref{8f16}. Then the equality of spaces
\begin{equation}\label{8f22}
H^{s-\lambda,(s-\lambda)/(2b);\varphi}(W)=
\bigl[H^{s_{0}-\lambda,(s_{0}-\lambda)/(2b)}(W),
H^{s_{1}-\lambda,(s_{1}-\lambda)/(2b)}(W)\bigr]_{\psi}
\end{equation}
holds true with equivalence of norms for arbitrary real $\lambda\leq s_{0}$ provided that $W=\Omega$ or $W=S$. If $W=\mathbb{R}^{k}$ with $2\leq k\in\mathbb{Z}$, then \eqref{8f22} holds true with equality of norms without the assumption that $0\leq s_{0}$.
\end{proposition}

This result is proved in \cite[Theorem~2 and Lemma~1]{Los16JMathSci} for
the cases where $W=S$ and  $W=\mathbb{R}^{k}$ respectively. In the $W=\Omega$ case, the proof of the result is the same as the proof of  its analog for a strip \cite[Lemma~2]{Los16JMathSci}. Note that we represent the indexes as $s-\lambda$ etc. for the sake of convenience of our application of Propositions \ref{8prop4} and \ref{8prop5} to the spaces used in \eqref{16rangeH}.

\section{Proofs}\label{16sec7}

To deduce Theorem \ref{16th4.1} from its known counterpart in the Sobolev case, we need to prove a version of Proposition \ref{8prop5} for the range of isomorphism \eqref{16f11}. This proof is based on the following lemma about properties of the operator that assigns the Cauchy data to an arbitrary function $g\in H^{s,s/(2b);\varphi}(S)$.

\begin{lemma}\label{8lem1}
Choose an integer $r\geq1$, and consider the linear mapping
\begin{equation}\label{8f25}
R:g\mapsto\bigl(g\!\upharpoonright\!\Gamma,
\partial_{t}g\!\upharpoonright\!\Gamma,\dots,
\partial^{r-1}_{t}g\!\upharpoonright\!\Gamma\bigr),
\quad\mbox{with}\quad g\in C^{\infty}(\overline{S}).
\end{equation}
This mapping extends uniquely (by continuity) to a bounded linear operator
\begin{equation}\label{8f65}
R:H^{s,s/(2b);\varphi}(S)\rightarrow \bigoplus_{k=0}^{r-1}
H^{s-2bk-b;\varphi}(\Gamma)=:\mathbb{H}^{s;\varphi}(\Gamma)
\end{equation}
for arbitrary $s>2br-b$ and $\varphi\in\mathcal{M}$. This operator is right invertible; moreover, there exists a continuous linear mapping  $T:(L_2(\Gamma))^r\to L_2(S)$ that the restriction of $T$ to the space $\mathbb{H}^{s;\varphi}(\Gamma)$
is a bounded linear operator
\begin{equation}\label{8f43}
T:\mathbb{H}^{s;\varphi}(\Gamma)\to H^{s,s/(2b);\varphi}(S)
\end{equation}
for all $s>2br-b$ and  $\varphi\in\mathcal{M}$ and that $RTv=v$ for every $v\in\mathbb{H}^{s;\varphi}(\Gamma)$.
\end{lemma}

\begin{proof}
We first prove an analog of this lemma for H\"ormander spaces
defined on $\mathbb{R}^{n}$ and $\mathbb{R}^{n-1}$ instead of $S$ and $\Gamma$. Then we deduce the lemma with the help of the special local charts on~$S$.

Consider the linear mapping
\begin{equation}\label{8f60}
R_{0}:w\mapsto\bigl(w\!\mid_{t=0},\,
\partial_{t}w\!\mid_{t=0},\dots,
\partial^{r-1}_{t}w\!\mid_{t=0}\bigr),\quad\mbox{with}
\quad w\in \mathcal{S}(\mathbb{R}^{n}).
\end{equation}
Here, we interpret $w$ as a function $w(x,t)$ of $x\in\mathbb{R}^{n-1}$ and $t\in\mathbb{R}$ so that $R_{0}w\in(\mathcal{S}(\mathbb{R}^{n-1}))^{r}$.
Choose $s>2br-b$ and $\varphi\in\mathcal{M}$ arbitrarily, and prove that  the mapping \eqref{8f60} extends uniquely (by continuity) to a bounded linear operator
\begin{equation}\label{8f63}
R_{0}:H^{s,s/(2b);\varphi}(\mathbb{R}^{n})\rightarrow \bigoplus_{k=0}^{r-1}
H^{s-2bk-b;\varphi}(\mathbb{R}^{n-1})=:
\mathbb{H}^{s;\varphi}(\mathbb{R}^{n-1}).
\end{equation}
This fact is known in the Sobolev case of $\varphi\equiv1$ due to \cite[Chapter~II, Theorem 7]{Slobodetskii58}. Using the interpolation with a function parameter between Sobolev spaces, we can deduce this fact in the general case of arbitrary $\varphi\in\mathcal{M}$.

Namely, choose $s_0,s_1\in\mathbb{R}$ such that $2br-b<s_0<s<s_1$, and
consider the bounded linear operators
\begin{equation}\label{8f63-sob}
R_{0}:H^{s_j,s_j/(2b)}(\mathbb{R}^{n})\to\mathbb{H}^{s_j}(\mathbb{R}^{n-1})
\quad\mbox{for each}\quad j\in\{0,1\}.
\end{equation}
Let $\psi$ be the interpolation parameter \eqref{8f16}. Then the restriction of the mapping \eqref{8f63-sob} with $j=0$ to the space
\begin{equation}\label{8f-intH}
\bigl[H^{s_0,s_{0}/(2b)}(\mathbb{R}^{n}),
H^{s_1,s_{1}/(2b)}(\mathbb{R}^{n})\bigr]_{\psi}=
H^{s,s/(2b);\varphi}(\mathbb{R}^{n})
\end{equation}
is a bounded operator
\begin{equation}\label{8f27}
R_{0}:H^{s,s/(2b);\varphi}(\mathbb{R}^{n})\to
\bigl[\mathbb{H}^{s_0}(\mathbb{R}^{n-1}),
\mathbb{H}^{s_1}(\mathbb{R}^{n-1})\bigr]_{\psi}.
\end{equation}
The latter equality is due to Proposition~\ref{8prop5}. This operator is an extension by continuity of the mapping \eqref{8f60} because the set
$\mathcal{S}(\mathbb{R}^{n})$ is dense in $H^{s,s/(2b);\varphi}(\mathbb{R}^{n})$. Owing to Propositions \ref{8prop2} and \ref{8prop4}, we get
\begin{equation}\label{8f-intHH}
\begin{aligned}
\bigl[\mathbb{H}^{s_0}(\mathbb{R}^{n-1}),
\mathbb{H}^{s_1}(\mathbb{R}^{n-1})\bigr]_{\psi}&=
\bigoplus_{k=0}^{r-1}
\bigl[H^{s_0-2bk-b}(\mathbb{R}^{n-1}),
H^{s_1-2bk-b}(\mathbb{R}^{n-1})\bigr]_{\psi}\\
&=\bigoplus_{k=0}^{r-1}H^{s-2bk-b;\varphi}(\mathbb{R}^{n-1})=
\mathbb{H}^{s;\varphi}(\mathbb{R}^{n-1}).
\end{aligned}
\end{equation}
Hence, the bounded linear operator \eqref{8f27} is the required operator \eqref{8f63}.

Let us now build a continuous linear mapping
\begin{equation}\label{8f58}
T_0:\bigl(L_{2}(\mathbb{R}^{n-1})\bigr)^r\to L_{2}(\mathbb{R}^{n})
\end{equation}
such that its restriction to every space $\mathbb{H}^{s;\varphi}(\mathbb{R}^{n-1})$, with $s>2br-b$ and $\varphi\in\mathcal{M}$, is a bounded operator between the spaces $\mathbb{H}^{s;\varphi}(\mathbb{R}^{n-1})$ and $H^{s,s/(2b);\varphi}(\mathbb{R}^{n})$ and that this operator is a right inverse of \eqref{8f63}.

Similarly to \cite[Proof of Theorem~2.5.7]{Hermander63}, we define the linear mapping
\begin{equation}\label{8f58-def}
T_0:v\mapsto F_{\xi\mapsto x}^{-1}\biggl[
\beta\bigl(\langle\xi\rangle^{2b}t\bigr)\,\sum_{k=0}^{r-1}
\frac{1}{k!}\,\widehat{v_k}(\xi)\times t^k\biggr](x,t)
\end{equation}
on the linear topological space of all vectors
$$
v:=(v_0,\dots,v_{r-1})\in\bigl(\mathcal{S}'(\mathbb{R}^{n-1})\bigr)^r.
$$
We consider $T_{0}v$ as a distribution on the Euclidean space $\mathbb{R}^{n}$ of points $(x,t)$, with $x=(x_{1},\ldots,x_{n-1})\in\mathbb{R}^{n-1}$ and $t\in\mathbb{R}$. In \eqref{8f58-def}, the function $\beta\in C^{\infty}_{0}(\mathbb{R})$ is chosen so that $\beta=1$ in a certain neighbourhood of zero. As usual, $F_{\xi\mapsto x}^{-1}$ denotes the inverse Fourier transform with respect to $\xi=(\xi_{1},\ldots,\xi_{n-1})\in\mathbb{R}^{n-1}$, and $\langle\xi\rangle:=(1+|\xi|^2)^{1/2}$. The variable $\xi$ is dual to $x$ relative to the direct Fourier transform $\widehat{w}(\xi)=(Fw)(\xi)$ of a function $w(x)$.

Obviously, the mapping \eqref{8f58-def} is well defined and acts continuously between $(\mathcal{S}'(\mathbb{R}^{n-1}))^r$ and $\mathcal{S}'(\mathbb{R}^{n})$. It is also evident that the restriction of this mapping to the space $(L_{2}(\mathbb{R}^{n-1}))^r$ is a continuous operator from $(L_{2}(\mathbb{R}^{n-1}))^r$ to $L_{2}(\mathbb{R}^{n})$.

We assert that
\begin{equation}\label{8f57}
R_{0}T_{0}v=v \quad\mbox{for every}\quad
v\in\bigl(\mathcal{S}(\mathbb{R}^{n-1})\bigr)^r.
\end{equation}
Since $v\in(\mathcal{S}(\mathbb{R}^{n-1}))^r$ implies $T_{0}v\in\mathcal{S}(\mathbb{R}^{n-1})$, the left-hand side of the equality \eqref{8f57} is well defined. Let us prove this equality.

Choosing $j\in\{0,\dots,r-1\}$ and
$v=(v_0,\dots,v_{r-1})\in(\mathcal{S}(\mathbb{R}^{n-1}))^r$
arbitrarily, we get
\begin{align*}
F\bigl[\partial^j_tT_0v\!\mid_{t=0}\bigr](\xi)&=
\partial^j_{t}F_{x\mapsto\xi}[T_0v](\xi,t)\big|_{t=0}=
\partial^j_t\biggl(
\beta\bigl(\langle\xi\rangle^{2b}t\bigr)\,\sum_{k=0}^{r-1}
\frac{1}{k!}\,\widehat{v_k}(\xi)\,t^k\biggr)\bigg|_{t=0}\\
&=\biggl(\partial^j_t\sum_{k=0}^{r-1}
\frac{1}{k!}\,\widehat{v_k}(\xi)\,t^k\biggr)\!
\bigg|_{t=0}=
\widehat{v_j}(\xi)
\end{align*}
for every $\xi\in\mathbb{R}^{n-1}$. In the third equality, we have used the fact that $\beta=1$ in a neighbourhood of zero.
Thus, the Fourier transforms of all the corresponding components of the vectors $R_{0}T_{0}v$ and $v$ coincide, which is equivalent to \eqref{8f57}.

Let us now prove that the restriction of the mapping \eqref{8f58-def} to each space
\begin{equation}\label{8f-doubleSobolev}
\mathbb{H}^{2bm}(\mathbb{R}^{n-1})=
\bigoplus_{k=0}^{r-1}H^{2bm-2bk-b}(\mathbb{R}^{n-1}),
\end{equation}
with $0\leq m\in\mathbb{Z}$, is a bounded operator between $\mathbb{H}^{2bm}(\mathbb{R}^{n-1})$ and $H^{2bm,m}(\mathbb{R}^{n})$.
Note that the integers $2bm-2bk-b$ may be negative in \eqref{8f-doubleSobolev}.

Let an integer $m\geq0$. We use the fact that the norm in the space $H^{2bm,m}(\mathbb{R}^{n})$ is equivalent to the norm
\begin{equation*}
\|w\|_{2bm,m}:=\|w\|+
\sum_{j=1}^{n-1}\|\partial_{x_j}^{2bm}w\|+
\|\partial_{t}^{m}w\|
\end{equation*}
(see, e.g., \cite[Section~9.1]{BesovIlinNikolskii75}). Here and below in this proof, $\|\cdot\|$ stands for the norm in the Hilbert space $L_2(\mathbb{R}^{n})$. Of course, $\partial_{x_j}$ and $\partial_{t}$ denote the operators of generalized differentiation with respect to $x_j$ and $t$ respectively. Choosing $v=(v_0,\dots,v_{r-1})\in(\mathcal{S}(\mathbb{R}^{n-1}))^r$ arbitrarily and using the Parseval equality, we obtain the following:
\begin{align*}
\|T_0v\|_{2bm,m}&=\|T_0v\|+
\sum_{j=1}^{n-1}\|\partial_{x_j}^{2bm}\,T_0v\|+
\|\partial_{t}^{m}\,T_0v\|\\
&=\|\widehat{T_0v}\|+
\sum_{j=1}^{n-1}\|\xi_j^{2bm}\,\widehat{T_0v}\|+
\|\partial_{t}^{m}\,\widehat{T_0v}\|\\
&\leq\sum_{k=0}^{r-1}\frac{1}{k!}\,
\biggl(\,\int\limits_{\mathbb{R}^{n}}
\bigl|\beta(\langle\xi\rangle^{2b}t)\,\widehat{v_k}(\xi)\,t^k\bigr|^2
d\xi dt\biggr)^{1/2}\\
&+\sum_{j=1}^{n-1}\sum_{k=0}^{r-1}\frac{1}{k!}\,
\biggl(\,\int\limits_{\mathbb{R}^{n}}
\bigl|\xi_j^{2bm}\,\beta(\langle\xi\rangle^{2b}t)\,\widehat{v_k}(\xi)\,
t^k\bigr|^2d\xi dt\biggr)^{1/2}\\
&+\sum_{k=0}^{r-1}\frac{1}{k!}\,
\biggl(\,\int\limits_{\mathbb{R}^{n}}
\bigl|\partial_{t}^{m}\bigl(\beta(\langle\xi\rangle^{2b}t)\,t^k\bigr)\,
\widehat{v_k}(\xi)\bigr|^2d\xi dt\biggr)^{1/2}.
\end{align*}

Let us estimate each of these three integrals separately. We begin with the third integral. Changing the variable $\tau=\langle\xi\rangle^{2b}t$ in the interior integral with respect to $t$, we get the equalities
\begin{align*}
\int\limits_{\mathbb{R}^{n}}
\bigl|\partial_{t}^{m}\bigl(\beta(\langle\xi\rangle^{2b}t)\,t^k\bigr)\,
\widehat{v_k}(\xi)\bigr|^2d\xi dt
&=\int\limits_{\mathbb{R}^{n-1}}
|\widehat{v_k}(\xi)|^2d\xi \int\limits_{\mathbb{R}}
|\partial_{t}^{m}(\beta(\langle\xi\rangle^{2b}t)t^k)|^2 dt\\
&=\int\limits_{\mathbb{R}^{n-1}}
\langle\xi\rangle^{4bm-4bk-2b}\,|\widehat{v_k}(\xi)|^2 d\xi \int\limits_{\mathbb{R}}
|\partial_{\tau}^{m}(\beta(\tau)\tau^{k})|^2 d\tau.
\end{align*}
Hence,
$$
\int\limits_{\mathbb{R}^{n}}
\bigl|\partial_{t}^{m}\bigl(\beta(\langle\xi\rangle^{2b}t)\,t^k\bigr)\,
\widehat{v_k}(\xi)\bigr|^2d\xi dt=
c_1\,\|v_k\|^2_{H^{2bm-2bk-b}(\mathbb{R}^{n-1})},
$$
with
$$
c_1:=\int\limits_{\mathbb{R}}
|\partial_{\tau}^{m}(\beta(\tau)\tau^{k})|^2 d\tau<\infty.
$$

Using the same changing of $t$ in the second integral, we obtain the following:
\begin{align*}
\int\limits_{\mathbb{R}^{n}}
\bigl|\xi_j^{2bm}\,\beta(\langle\xi\rangle^{2b}t)\,\widehat{v_k}(\xi)\,
t^k\bigr|^2d\xi dt
&=\int\limits_{\mathbb{R}^{n-1}}
|\xi_j|^{4bm}|\widehat{v_k}(\xi)|^2d\xi \int\limits_{\mathbb{R}}
|t^{k}\,\beta(\langle\xi\rangle^{2b}t)|^2dt\\
&=\int\limits_{\mathbb{R}^{n-1}}
|\xi_j|^{4bm}\langle\xi\rangle^{-4bk-2b}\,|\widehat{v_k}(\xi)|^2d\xi \int\limits_{\mathbb{R}}|\tau^{k}\beta(\tau)|^2d\tau\\
&\leq\int\limits_{\mathbb{R}^{n-1}}
\langle\xi\rangle^{4bm-4bk-2b}\,|\widehat{v_k}(\xi)|^2d\xi
\int\limits_{\mathbb{R}}|\tau^{k}\beta(\tau)|^2d\tau.
\end{align*}
Hence,
$$
\int\limits_{\mathbb{R}^{n}}
\bigl|\xi_j^{2bm}\,\beta(\langle\xi\rangle^{2b}t)\,\widehat{v_k}(\xi)\,
t^k\bigr|^2d\xi dt
\leq c_2\,\|v_k\|^2_{H^{2bm-2bk-b}(\mathbb{R}^{n-1})},
$$
with
$$
c_2:=\int\limits_{\mathbb{R}}|\tau^{k}\beta(\tau)|^2d\tau<\infty.
$$

The first integral is estimated analogously:
\begin{align*}
\int\limits_{\mathbb{R}^{n}}
\bigl|\beta(\langle\xi\rangle^{2b}t)\,\widehat{v_k}(\xi)\,
t^k\bigr|^2d\xi dt
&=\int\limits_{\mathbb{R}^{n-1}}
\langle\xi\rangle^{-4bk-2b}\,|\widehat{v_k}(\xi)|^2d\xi
\int\limits_{\mathbb{R}}|\tau^{k}\beta(\tau)|^2d\tau\\
&=c_2\,\|v_k\|^2_{H^{-2bk-b}(\mathbb{R}^{n-1})}
\leq c_2\,\|v_k\|^2_{H^{2bm-2bk-b}(\mathbb{R}^{n-1})}.
\end{align*}

Thus, we conclude that
\begin{equation*}
\|T_0v\|_{H^{2bm,m}(\mathbb{R}^{n})}^{2}\leq c\,\sum_{k=0}^{r-1}
\|v_k\|^2_{H^{2bm-2bk-b}(\mathbb{R}^{n-1})}=
c\,\|v\|_{\mathbb{H}^{2bm}(\mathbb{R}^{n-1})}^{2}
\end{equation*}
for any $v\in(\mathcal{S}(\mathbb{R}^{n-1}))^r$, with the number $c>0$ being independent of $v$. Since the set $\bigl(S(\mathbb{R}^{n-1})\bigr)^r$ is dense in $\mathbb{H}^{2bm}(\mathbb{R}^{n-1})$, it follows from the latter estimate that the mapping \eqref{8f58-def} sets a bounded linear operator
\begin{equation*}
T_0:\mathbb{H}^{2bm}(\mathbb{R}^{n-1})\to H^{2bm,m}(\mathbb{R}^{n})
\quad\mbox{whenever}\quad 0\leq m\in\mathbb{Z}.
\end{equation*}

Let us deduce from this fact that the mapping \eqref{8f58-def} acts continuously between the spaces $\mathbb{H}^{s;\varphi}(\mathbb{R}^{n-1})$ and $H^{s,s/(2b);\varphi}(\mathbb{R}^{n})$ for every $s>2br-b$ and $\varphi\in\mathcal{M}$. Put $s_0=0$, and choose an integer $s_1>s$ such that $s_1/(2b)\in\mathbb{Z}$,
and consider the linear bounded operators
\begin{equation}\label{8f66}
T_{0}:\mathbb{H}^{s_j}(\mathbb{R}^{n-1})\to H^{s_j,s_j/(2b)}(\mathbb{R}^{n}),
\quad\mbox{with}\quad j\in\{0,1\}.
\end{equation}
Let, as above, $\psi$ be the interpolation parameter \eqref{8f16}. Then the restriction of the mapping \eqref{8f66} with $j=0$ to the space
\begin{equation*}
\bigl[\mathbb{H}^{s_0}(\mathbb{R}^{n-1}),
\mathbb{H}^{s_1}(\mathbb{R}^{n-1})\bigr]_{\psi}=
\mathbb{H}^{s;\varphi}(\mathbb{R}^{n-1})
\end{equation*}
is a bounded operator
\begin{equation}\label{8f48}
T_0:\mathbb{H}^{s;\varphi}(\mathbb{R}^{n-1})\to
H^{s,s/(2b);\varphi}(\mathbb{R}^{n}).
\end{equation}
Here, we have used formulas \eqref{8f-intH} and \eqref{8f-intHH}, which remain true for the considered $s_{0}$ and $s_{1}$.

Now the equality \eqref{8f57} extends by continuity over all vectors
$v\in\mathbb{H}^{s;\varphi}(\mathbb{R}^{n-1})$. Hence, the operator \eqref{8f48} is right inverse to \eqref{8f63}. Thus, the required mapping \eqref{8f58} is built.

We need to introduce analogs of the operators \eqref{8f63} and \eqref{8f48} for the strip
$$
\Pi=\bigl\{(x,t):x\in\mathbb{R}^{n-1},0<t<\tau\bigr\}.
$$
Let $s>2br-b$ and $\varphi\in\mathcal{M}$. Given $u\in H^{s,s/(2b);\varphi}(\Pi)$, we put
$R_{1}u:=R_{0}w$, where a function
$w\in H^{s,s/(2b);\varphi}(\mathbb{R}^{n})$ satisfies the condition
$w\!\upharpoonright\!\Pi=u$. Evidently, this definition does not depend on the choice of $w$. The linear mapping $u\mapsto R_{1}u$ is a bounded  operator
\begin{equation}\label{8f67}
R_{1}:H^{s,s/(2b);\varphi}(\Pi)\to\mathbb{H}^{s;\varphi}(\mathbb{R}^{n-1}).
\end{equation}
This follows immediately from the boundedness of the operator \eqref{8f63} and from the definition of $H^{s,s/(2b);\varphi}(\Pi)$.

Let us build a right inverse of \eqref{8f67} on the base of the mapping \eqref{8f58-def}. We put $T_{1}v:=(T_0v)\!\upharpoonright\!\Pi$ for arbitrary $v\in(L_{2}(\mathbb{R}^{n-1}))^{r}$. Owing to \eqref{8f58},
the linear mapping $v\mapsto T_{1}v$ acts continuously between $(L_{2}(\mathbb{R}^{n-1}))^{r}$ and $L_{2}(\Pi)$. Moreover, its restriction to $\mathbb{H}^{s;\varphi}(\mathbb{R}^{n-1})$ is a bounded operator
\begin{equation}\label{8f51}
T_1:\mathbb{H}^{s;\varphi}(\mathbb{R}^{n-1})\to H^{s,s/(2b);\varphi}(\Pi).
\end{equation}
This follows directly from the boundedness of the operator \eqref{8f48}. Besides,
$$
R_1T_1v=R_1\bigl((T_0v)\!\upharpoonright\!\Pi\bigr)=R_0T_0v=v
\quad\mbox{for every}\quad v\in\mathbb{H}^{s;\varphi}(\mathbb{R}^{n-1}).
$$
Thus, the operator \eqref{8f51} is a right inverse of \eqref{8f67}.

Using operators \eqref{8f67} and \eqref{8f51}, we can now prove our lemma with the help of the special local charts \eqref{8f-local} on $S$. As above, let $s>2br-b$ and $\varphi\in\mathcal{M}$. Choosing $g\in C^{\infty}(\overline{S})$ arbitrarily, we get the following:
\begin{align*}
\|Rg\|_{\mathbb{H}^{s;\varphi}(\Gamma)}^{2}
&=\sum_{k=0}^{r-1}\|\partial^{k}_{t}g\!\upharpoonright\!\Gamma\|_
{H^{s-2bk-b;\varphi}(\Gamma)}^{2}\\
&=\sum_{k=0}^{r-1}\sum_{j=1}^{\lambda}
\|(\chi_{j}(\partial^{k}_{t}g\!\upharpoonright\!\Gamma))
\circ\theta_{j}\|_{H^{s-2bk-b;\varphi}(\mathbb{R}^{n-1})}^{2}\\
&=\sum_{j=1}^{\lambda}\sum_{k=0}^{r-1}
\|\partial^{k}_{t}((\chi_{j}\,g)\circ\theta^{\ast}_{j})
\!\upharpoonright\!\mathbb{R}^{n-1})\|_
{H^{s-2bk-b;\varphi}(\mathbb{R}^{n-1})}^{2}\\
&\leq c^{2}\,\sum_{j=1}^{\lambda}
\|(\chi_{j}\,g)\circ\theta^{\ast}_{j}\|_{H^{s,s/(2b);\varphi}(\Pi)}^{2}=
c^{2}\,\|g\|_{H^{s,s/(2b);\varphi}(S)}^{2}.
\end{align*}
Here, $c$ denotes the norm of the bounded operator \eqref{8f67}, and, as usual, the symbol "$\circ$" stands for a composition of functions. Recall that $\{\theta_{j}\}$ is a collection of local charts on $\Gamma$ and that $\{\chi_{j}\}$ is an infinitely smooth partition of unity on $\Gamma$. Thus,
\begin{equation*}
\|Rg\|_{\mathbb{H}^{s;\varphi}(\Gamma)}\leq c\,\|g\|_{H^{s,s/(2b);\varphi}(S)}
\quad\mbox{for every}\quad g\in C^{\infty}(\overline{S}).
\end{equation*}
This implies that the mapping \eqref{8f25} extends by continuity to the bounded linear operator \eqref{8f65}.

Let us now build a continuous linear mapping $T:(L_2(\Gamma))^r\to L_2(S)$ whose restriction to $\mathbb{H}^{s;\varphi}(\Gamma)$ is a right inverse of \eqref{8f65}. Consider the linear mapping of flattening of $\Gamma$
\begin{equation*}
L:\omega\mapsto\bigl((\chi_{1}\omega)\circ\theta_{1},\ldots,
(\chi_{\lambda}\omega)\circ\theta_{\lambda}\bigr),
\quad\mbox{with}\quad\omega\in L_2(\Gamma).
\end{equation*}
This mapping acts continuously between $L_{2}(\Gamma)$ and $(L_{2}(\mathbb{R}^{n-1}))^{\lambda}$. Moreover, its restriction to $H^{\sigma;\varphi}(\Gamma)$ is an isometric operator
\begin{equation}\label{8f52}
L:H^{\sigma;\varphi}(\Gamma)\rightarrow
\bigl(H^{\sigma;\varphi}(\mathbb{R}^{n-1})\bigr)^{\lambda}
\quad\mbox{whenever}\quad\sigma>0.
\end{equation}
Besides, consider the linear mapping of sewing of $\Gamma$
\begin{equation*}
K:(w_{1},\ldots,w_{\lambda})\mapsto\sum_{j=1}^{\lambda}\,
O_{j}\bigl((\eta_{j}w_{j})\circ\theta_{j}^{-1}\bigr),
\quad\mbox{with}\quad w_{1},\ldots,w_{\lambda}\in L_2(\mathbb{R}^{n-1}).
\end{equation*}
Here, each function $\eta_{j}\in
C_{0}^{\infty}(\mathbb{R}^{n-1})$ is chosen so that $\eta_{j}=1$ on the
set $\theta^{-1}_{j}(\mathrm{supp}\,\chi_{j})$, whereas $O_{j}$
denotes the operator of the extension by zero to $\Gamma$ of a function given on $\Gamma_j$. The restriction of this mapping to $(H^{\sigma;\varphi}(\mathbb{R}^{n-1}))^{\lambda}$ is a bounded  operator
\begin{equation*}
K:\bigl(H^{\sigma;\varphi}(\mathbb{R}^{n-1})\bigr)^{\lambda}\to
H^{\sigma;\varphi}(\Gamma)\quad\mbox{whenever}\quad\sigma>0,
\end{equation*}
and this operator is left inverse to \eqref{8f52} (see \cite[the proof of Theorem~2.2]{MikhailetsMurach14}).

The mapping $K$ induces the operator $K_{1}$ of the sewing of the manifold $S=\Gamma\times(0,\tau)$ by the formula
\begin{equation*}
\bigl(K_1(u_1,\dots,u_\lambda)\bigr)(x,t):=
\bigl(K(u_1(\cdot,t),\ldots,u_\lambda(\cdot,t))\bigr)(x)
\end{equation*}
for arbitrary functions $u_1,\dots,u_\lambda\in L_2(\Pi)$ and almost all $x\in\Gamma$ and $t\in(0,\tau)$. The linear operator $K_1$ acts continuously between $(L_2(\Pi))^{\lambda}$ and $L_2(S)$. Moreover, its restriction to $(H^{\sigma,\sigma/(2b);\varphi}(\Pi))^{\lambda}$ is a bounded  operator
\begin{equation}\label{8f53}
K_{1}:(H^{\sigma,\sigma/(2b);\varphi}(\Pi))^{\lambda}\to
H^{\sigma,\sigma/(2b);\varphi}(S)\quad\mbox{whenever}\quad\sigma>0
\end{equation}
(see \cite[the proof of Theorem~2]{Los16JMathSci}).

Given $v:=(v_0,v_1,\dots,v_{r-1})\in(L_{2}(\Gamma))^{r}$, we set
\begin{equation*}
Tv:=K_1\bigl(T_1(v_{0,1},\ldots,v_{r-1,1}),\ldots,
T_1(v_{0,\lambda},\ldots,v_{r-1,\lambda})\bigr),
\end{equation*}
where
$$
(v_{k,1},\ldots,v_{k,\lambda}):=
Lv_{k}\in(L_{2}(\mathbb{R}^{n-1}))^{\lambda}
$$
for each integer $k\in\{0,\ldots,r-1\}$. The linear mapping $v\mapsto Tv$ acts continuously between $(L_{2}(\Gamma))^{r}$ and $L_{2}(S)$, which follows directly from the corresponding properties of $L$, $T_1$, and~$K_1$. The restriction of this mapping to $\mathbb{H}^{s;\varphi}(\Gamma)$ is the required bounded operator \eqref{8f43}. Indeed, its boundedness follows immediately from the boundedness of the operators \eqref{8f51}, \eqref{8f52}, and \eqref{8f53}.
Besides, the operator \eqref{8f43} is a right inverse of \eqref{8f65} because
\begin{align*}
(RTv)_k&=\bigl(RK_1\bigl(T_1(v_{0,1},\ldots,v_{r-1,1}),\ldots,
T_1(v_{0,\lambda},\ldots,v_{r-1,\lambda})\bigr)\bigr)_k\\
&=K\bigl(\bigl(R_1T_1(v_{0,1},\ldots,v_{r-1,1})\bigr)_k,\ldots,
\bigl(R_1T_1(v_{0,\lambda},\ldots,v_{r-1,\lambda})\bigr)_k\bigr)\\
&=K(v_{k,1},\dots,v_{k,\lambda})=KLv_k=v_k
\end{align*}
for an arbitrary  vector $v=(v_0,v_1,\dots,v_{r-1})\in\mathbb{H}^{s;\varphi}(\Gamma)$.
Here, the index $k$ runs over the set $\{0,\dots,r-1\}$ and denotes the $k$-th component of a vector.
\end{proof}

Using this lemma, we will now prove a version of Proposition \ref{8prop5} for the range of  isomorphism \eqref{16f11}.
It is sufficient to restrict ourselves to the $s\notin E$ case; recall that $E$ is denoted by \eqref{setE}.
Let $\{J_l:1\leq l\in\mathbb{Z}\}$ stand for the collection of all
connected components of the set $(\sigma_0,\infty)\setminus E$. Each  component $J_l$ is a certain finite subinterval of $(\sigma_0,\infty)$.

\begin{lemma}\label{16lem7.4}
Let $1\leq l\in\mathbb{Z}$. Suppose that real numbers $s_0,s,s_1\in J_{l}$ satisfy the inequality $s_0<s<s_1$ and that $\varphi\in\mathcal{M}$. Define an interpolation parameter $\psi\in\mathcal{B}$ by formula \eqref{8f16}.
Then the equality of spaces
\begin{equation}\label{16f42}
\mathcal{Q}^{s-2m,(s-2m)/(2b);\varphi}=\,
\bigl[\mathcal{Q}^{s_0-2m,(s_0-2m)/(2b)},
\mathcal{Q}^{s_1-2m,(s_1-2m)/(2b)}\bigr]_{\psi}
\end{equation}
holds true up to equivalence of norms.
\end{lemma}

\begin{proof}
It relies on Proposition~\ref{8prop1} and the interpolation formula
for the pair of Sobolev spaces $\mathcal{H}^{s_0-2m,(s_0-2m)/(2b)}$
and $\mathcal{H}^{s_1-2m,(s_1-2m)/(2b)}$. Bearing in mind the compatibility conditions \eqref{16f8}, consider the set
\begin{equation}\label{set-kj}
\biggl\{k\in\mathbb{Z}:0\leq k<\frac{s-m_j-1/2-b}{2b}\biggr\}
\end{equation}
for each $j\in\{1,\dots,m\}$. This set
does not depend on $s\in J_l$. Let
$q_{l,j}^{\star}$ denote the number of all elements of
\eqref{set-kj}, and put $q_{l,j}:=q_{l,j}^{\star}-1$ for convenience.

Let us build a linear mapping $P_l$ on
\begin{equation}\label{16f43-aa}
\bigcup_{\sigma\in J_{l}}\mathcal{H}^{\sigma-2m,(\sigma-2m)/(2b)}
\end{equation}
that its restriction to the space
$\mathcal{H}^{\sigma-2m,(\sigma-2m)/(2b)}$ is a projector of this space on its subspace $\mathcal{Q}^{\sigma-2m,(\sigma-2m)/(2b)}$ for every
$\sigma\in J_{l}$. Choosing a vector
$$
F:=\bigl(f,g_1,\dots,g_m,h_0,\dots,h_{\varkappa-1}\bigr)\in
\bigcup_{\sigma\in J_{l}}\mathcal{H}^{\sigma-2m,(\sigma-2m)/(2b)}
$$
arbitrarily, we put
\begin{equation}\label{16f43}
\left\{
  \begin{array}{ll}
    g^{*}_{j}:=g_j & \hbox{whenever}\; q_{l,j}=-1,\\
    g^{*}_{j}:=g_j+T_{l,j}(w_{j,0},\dots,w_{j,q_{l,j}})
    & \hbox{whenever}\;q_{l,j}\geq0
  \end{array}
\right.
\end{equation}
for each $j\in\{1,\dots,m\}$. Here,
\begin{align*}
w_{j,0}&:=B_{j,0}(v_{0},\dots,v_{[m_j/(2b)]})\!\upharpoonright\!\Gamma-g_j\!\upharpoonright\!\Gamma,\\
&\dots\\
w_{j,q_{l,j}}&:=B_{j,q_{l,j}}(v_{0},\dots,v_{[m_j/(2b)]+q_{l,j}})\!\upharpoonright\!\Gamma-
\partial_{t}^{q_{l,j}} g_j\!\upharpoonright\!\Gamma,
\end{align*}
with the functions $v_0,v_1\dots,v_{[m_j/(2b)]+q_{l,j}}$ and the differential operators $B_{j,0},\ldots,B_{j,q_{l,j}}$ being defined
by \eqref{16f9} and \eqref{16f9B} respectively and with $T_{l,j}$ denoting the linear mapping $T$ from Lemma~\ref{8lem1} in the $r=q_{l,j}^\star$ case.

The linear mapping
\begin{equation*}
P_l:\,\bigl(f,g_1,\dots,g_m,h_0,\dots,h_{\varkappa-1}\bigr)\mapsto
\bigl(f,g^*_1,\dots,g^*_m,h_0,\dots,h_{\varkappa-1}\bigr)
\end{equation*}
given on the space \eqref{16f43-aa} is required. Indeed, its restriction to the space $\mathcal{H}^{\sigma-2m,(\sigma-2m)/(2b)}$ is a bounded operator on this space for every $\sigma\in J_l$, which follows from
\eqref{16f9}--\eqref{8f69aa} and \eqref{8f43}. Note that we
use the boundedness of the operators \eqref{8f43} in the case where
$r=q_{l,j}^\star$, $s=\sigma-m_j-1/2$, and $\varphi(\cdot)\equiv1$,
with the condition $s>2br-b$ being satisfied because
\begin{equation*}
q_{l,j}^\star<\frac{\sigma-m_j-1/2-b}{2b}+1\quad\mbox{whenever}\quad
\sigma\in J_l.
\end{equation*}
Besides, it follows from the definition of $P_l$ and the compatibility
conditions \eqref{16f8} that
$P_lF\in\mathcal{Q}^{\sigma-2m,(\sigma-2m)/(2b)}$ for every
$F\in\mathcal{H}^{\sigma-2m,(\sigma-2m)/(2b)}$. Indeed, these conditions for the vector $P_{l}F=(f,g^*_1,\dots,g^*_m,h_0,\dots,h_{\varkappa-1})$ become
\begin{equation*}
\partial^{k}_t g_j^*\!\upharpoonright\!\Gamma=
B_{j,k}(v_0,\dots,v_{[m_j/(2b)]+k})\!\upharpoonright\!\Gamma
\end{equation*}
for all $j$ and $k$ indicated in \eqref{16f8}. However,
\begin{align*}
\partial^{k}_t g_j^*\!\upharpoonright\!\Gamma=&
\partial^{k}_t g_j\!\upharpoonright\!\Gamma+
\partial^{k}_t T_{l,j}(w_{j,0},\dots,w_{j,q_{l,j}})\!\upharpoonright\!\Gamma\\
=& \partial^{k}_t g_j\!\upharpoonright\!\Gamma+w_{j,k}=
B_{j,k}(v_0,\dots,v_{[m_j/(2b)]+k})\!\upharpoonright\!\Gamma
\end{align*}
provided that $q_{l,j}\geq0$. If $q_{l,j}=-1$, there will not be the compatibility conditions involving~$g_{j}$. Thus, the vector $P_{l}F$ satisfies the compatibility conditions, i.e. $P_{l}F\in\mathcal{Q}^{\sigma-2m,(\sigma-2m)/(2b)}$. Moreover, $F\in\mathcal{Q}^{\sigma-2m,(\sigma-2m)/(2b)}$ implies that
$P_lF=F$. Namely, if $F\in\mathcal{Q}^{\sigma-2m,(\sigma-2m)/(2b)}$, then
\eqref{16f8} holds, which entails that all $w_{j,k}=0$, i.e.
$g_1^*=g_1$, ..., $g_m^*=g_m$.

Now we use Proposition~\ref{8prop1} in which
$X_j:=\mathcal{H}^{s_j-2m,(s_j-2m)/(2b)}$ and
$Y_j:=\mathcal{Q}^{s_j-2m,(s_j-2m)/(2b)}$ for each $j\in\{0,1\}$,
and $P:=P_l$. According to this proposition, the pair
$$
\bigl[\mathcal{Q}^{s_0-2m,(s_0-2m)/(2b)},
\mathcal{Q}^{s_1-2m,(s_1-2m)/(2b)}\bigr]
$$
is regular, and
\begin{equation}\label{16f44}
\begin{aligned}
&\bigl[\mathcal{Q}^{s_0-2m,(s_0-2m)/(2b)},
\mathcal{Q}^{s_1-2m,(s_1-2m)/(2b)}\bigr]_{\psi}\\
&=\bigl[\mathcal{H}^{s_0-2m,(s_0-2m)/(2b)},
\mathcal{H}^{s_1-2m,(s_1-2m)/(2b)}\bigr]_{\psi}\cap
\mathcal{Q}^{s_0-2m,(s_0-2m)/(2b)}.
\end{aligned}
\end{equation}
The right-hand side of this equality is a subspace of
$$
\bigl[\mathcal{H}^{s_0-2m,(s_0-2m)/(2b)},
\mathcal{H}^{s_1-2m,(s_1-2m)/(2b)}\bigr]_{\psi}.
$$
Owing to Propositions \ref{8prop2}, \ref{8prop4}, and \ref{8prop5},
we obtain the following equalities:
\begin{align*}
\bigl[&\mathcal{H}^{s_0-2m,(s_0-2m)/(2b)},
\mathcal{H}^{s_1-2m,(s_1-2m)/(2b)}\bigr]_{\psi} \\
&=\bigl[H^{s_0-2m,(s_0-2m)/(2b)}(\Omega),H^{s_1-2m,(s_1-2m)/(2b)}(\Omega)\bigr]_{\psi}\\
&\qquad
\oplus\bigoplus_{j=1}^{m}
\bigl[H^{s_0-m_j-1/2,\,(s_0-m_j-1/2)/(2b)}(S),
H^{s_1-m_j-1/2,\,(s_1-m_j-1/2)/(2b)}(S)\bigr]_{\psi}\\
&\qquad
\oplus\bigoplus_{k=0}^{\varkappa-1}\bigl[H^{s_0-2bk-b}(G),H^{s_1-2bk-b}(G)\bigr]_{\psi}\\
&=H^{s-2m,(s-2m)/(2b);\varphi}(\Omega)
\oplus\bigoplus_{j=1}^{m}H^{s-m_j-1/2,(s-m_j-1/2)/(2b);\varphi}(S)
\oplus\bigoplus_{k=0}^{\varkappa-1}H^{s-2bk-b;\varphi}(G)\\
&=\mathcal{H}^{s-2m,(s-2m)/(2b);\varphi}.
\end{align*}
Thus,
\begin{equation}\label{16f45}
\bigl[\mathcal{H}^{s_0-2m,(s_0-2m)/(2b)},
\mathcal{H}^{s_1-2m,(s_1-2m)/(2b)}\bigr]_{\psi}=
\mathcal{H}^{s-2m,(s-2m)/(2b);\varphi}
\end{equation}
up to equivalence of norms.
Formulas \eqref{16f44} and \eqref{16f45} give
\begin{align*}
&\bigl[\mathcal{Q}^{s_0-2m,(s_0-2m)/(2b)},
\mathcal{Q}^{s_1-2m,(s_1-2m)/(2b)}\bigr]_{\psi}\\
&=\mathcal{H}^{s-2m,(s-2m)/(2b);\varphi}\cap
\mathcal{Q}^{s_0-2m,(s_0-2m)/(2b)}=
\mathcal{Q}^{s-2m,(s-2m)/(2b);\varphi}.
\end{align*}
The latter equality holds because $s,s_0\in J_l$, i.e.
the elements of the subspace $\mathcal{Q}^{s-2m,(s-2m)/(2b);\varphi}$
satisfy the same compatibility conditions as the elements of
$\mathcal{Q}^{s_0-2m,(s_0-2m)/(2b)}$.
\end{proof}

Lemma \ref{16lem7.4} just proved will be used in the proof of Theorem~\ref{16th4.1} in the $s\notin E$ case. Examining
the opposite case, we need the following two results.

\begin{lemma}\label{16lem7.5}
Let numbers $s,\varepsilon\in\mathbb{R}$ satisfy
$s>\varepsilon>0$, and let $\varphi\in\mathcal{M}$.
Then the equality of spaces
\begin{equation}\label{16f46}
H^{s,s/(2b);\varphi}(W)=
\bigl[H^{s-\varepsilon,(s-\varepsilon)/(2b);\varphi}(W),
H^{s+\varepsilon,(s+\varepsilon)/(2b);\varphi}(W)\bigr]_{1/2}
\end{equation}
holds true up to equivalence of norms provided that $W=\Omega$ or $W=S$.
\end{lemma}

\begin{proof}
Choose a number $\delta>0$ such that $s-\varepsilon-\delta>0$.
According to Proposition~\ref{8prop5} for $\lambda=0$, we have the equalities
\begin{equation*}
H^{s-\varepsilon,(s-\varepsilon)/(2b);\varphi}(W)=
\bigl[H^{s-\varepsilon-\delta,(s-\varepsilon-\delta)/(2b)}(W),
H^{s+\varepsilon+\delta,(s+\varepsilon+\delta)/(2b)}(W)\bigr]_{\alpha}
\end{equation*}
and
\begin{equation*}
H^{s+\varepsilon,(s+\varepsilon)/(2b);\varphi}(W)=
\bigl[H^{s-\varepsilon-\delta,(s-\varepsilon-\delta)/(2b)}(W),
H^{s+\varepsilon+\delta,(s+\varepsilon+\delta)/(2b)}(W)\bigr]_{\beta}.
\end{equation*}
Here, the interpolation parameters $\alpha$ and $\beta$
are defined by the formulas
\begin{equation*}
\alpha(r):=r^{\delta/(2\varepsilon+2\delta)}
\varphi(r^{1/(2\varepsilon+2\delta)})
\quad\mbox{and}
\quad
\beta(r):=r^{(2\varepsilon+\delta)/(2\varepsilon+2\delta)}
\varphi(r^{1/(2\varepsilon+2\delta)})
\quad\mbox{if}\quad r\geq1,
\end{equation*}
and $\alpha(r)=\beta(r):=1$ if $0<r<1$. Owing to Propositions \ref{8prop3} and \ref{8prop5}, we then get
\begin{align*}
\bigl[&H^{s-\varepsilon,(s-\varepsilon)/(2b);\varphi}(W),
H^{s+\varepsilon,(s+\varepsilon)/(2b);\varphi}(W)\bigr]_{1/2}\\
\notag
&=\Bigl[
\bigl[H^{s-\varepsilon-\delta,(s-\varepsilon-\delta)/(2b)}(W),
H^{s+\varepsilon+\delta,(s+\varepsilon+\delta)/(2b)}(W)\bigr]_{\alpha},\\
&\quad\quad\bigl[H^{s-\varepsilon-\delta,(s-\varepsilon-\delta)/(2b)}(W),
H^{s+\varepsilon+\delta,(s+\varepsilon+\delta)/(2b)}(W)\bigr]_{\beta}
\Bigr]_{1/2}\\ \label{8f74}
&=\bigl[H^{s-\varepsilon-\delta,(s-\varepsilon-\delta)/(2b)}(W),
H^{s+\varepsilon+\delta,(s+\varepsilon+\delta)/(2b)}(W)\bigr]_{\omega}=
H^{s,s/(2b);\varphi}(W).
\end{align*}
Here, the interpolation parameter $\omega$ is defined by the formulas
\begin{equation*}
\omega(r):=\alpha(r)(\beta(r)/\alpha(r))^{1/2}=r^{1/2}\varphi(r^{1/(2\varepsilon+2\delta)})
\quad\mbox{if}\quad r\geq1
\end{equation*}
and $\omega(r):=1$ if $0<r<1$. As to Proposition \ref{8prop5}, note that
the interpolation parameter $\omega$ equals the right-hand side of \eqref{8f16} if we put $s_0:=s-\varepsilon-\delta$ and $s_1:=s+\varepsilon+\delta$.
Thus, \eqref{16f46} is valid.
\end{proof}

\begin{lemma}\label{16lem7.6}
Let numbers $s\in\mathbb{R}$ and $\varepsilon>0$ satisfy
$s-\varepsilon>\sigma_0$, and let $\varphi\in\mathcal{M}$.
Then the equality of spaces
\begin{equation}\label{16f67}
\mathcal{H}^{s-2m,(s-2m)/(2b);\varphi}=
\bigl[\mathcal{H}^{s-\varepsilon-2m,(s-\varepsilon-2m)/(2b);\varphi},
\mathcal{H}^{s+\varepsilon-2m,(s+\varepsilon-2m)/(2b);\varphi}\bigr]_{1/2}
\end{equation}
holds true up to equivalence of norms.
\end{lemma}

\begin{proof}
Relation \eqref{16f67} follows from Proposition~\ref{8prop2} and formula \eqref{16f46} and its analog for the isotropic space $H^{\sigma;\varphi}(G)$, with $\sigma>0$ (see \cite[Lemma~4.3]{MikhailetsMurach14}).
Indeed,
\begin{align*}
\bigl[&\mathcal{H}^{s-\varepsilon-2m,(s-\varepsilon-2m)/(2b);\varphi},
\mathcal{H}^{s+\varepsilon-2m,(s+\varepsilon-2m)/(2b);\varphi}\bigr]_{1/2} \\&=
\biggl[H^{s-\varepsilon-2m,(s-\varepsilon-2m)/(2b);\varphi}(\Omega)\oplus
\bigoplus_{j=1}^{m}
H^{s-\varepsilon-m_j-1/2,(s-\varepsilon-m_j-1/2)/(2b);\varphi}(S)\\
&\qquad\oplus\bigoplus_{k=0}^{\varkappa-1}
H^{s-\varepsilon-2bk-b;\varphi}(G), \\
&\quad\;\;\, H^{s+\varepsilon-2m,(s+\varepsilon-2m)/(2b);\varphi}(\Omega)\oplus
\bigoplus_{j=1}^{m}H^{s+\varepsilon-m_j-1/2,(s+\varepsilon-m_j-1/2)/(2b);
\varphi}(S)\\
&\qquad\oplus\bigoplus_{k=0}^{\varkappa-1}
H^{s+\varepsilon-2bk-b;\varphi}(G)\biggr]_{1/2}\\
&=\bigl[H^{s-2m-\varepsilon,(s-2m-\varepsilon)/(2b);\varphi}(\Omega),
H^{s-2m+\varepsilon,(s-2m+\varepsilon)/(2b);\varphi}(\Omega)\bigr]_{1/2}\\
&\qquad
\oplus\bigoplus_{j=1}^{m}
\bigl[H^{s-m_j-1/2-\varepsilon,\,(s-m_j-1/2-\varepsilon)/(2b);\varphi}(S),
H^{s-m_j-1/2+\varepsilon,\,(s-m_j-1/2+\varepsilon)/(2b);\varphi}(S)\bigr]_{1/2}\\
&\qquad
\oplus\bigoplus_{k=0}^{\varkappa-1}\bigl[H^{s-2bk-b-\varepsilon;\varphi}(G),
H^{s-2bk-b+\varepsilon;\varphi}(G)\bigr]_{1/2}\\
&=H^{s-2m,(s-2m)/(2b);\varphi}(\Omega)
\oplus\bigoplus_{j=1}^{m}H^{s-m_j-1/2,(s-m_j-1/2)/(2b);\varphi}(S)
\oplus\bigoplus_{k=0}^{\varkappa-1}H^{s-2bk-b;\varphi}(G)\\
&=\mathcal{H}^{s-2m,(s-2m)/(2b);\varphi}.
\end{align*}
\end{proof}

Now we are in position to prove Theorem~\ref{16th4.1}.

\begin{proof}[Proof of Theorem $\ref{16th4.1}$.]
Let $s>\sigma_0$ and $\varphi\in\mathcal{M}$.
We first consider the case where $s\notin E$.
Then $s\in J_{l}$ for a certain integer $l\geq1$.
Choose numbers $s_0,s_1\in J_{l}$ such that $s_0<s<s_1$ and that  $s_j+1/2\notin\mathbb{Z}$ and $s_j/(2b)+1/2\notin\mathbb{Z}$ whenever $j\in\{0,1\}$. According to Zhitarashu \cite[Theorem~9.1]{Zhitarashu85}, the mapping \eqref{16f4}
extends uniquely (by continuity) to an isomorphism
\begin{equation}\label{16f49}
\Lambda:\,H^{s_j,s_j/(2b)}(\Omega)\leftrightarrow
\mathcal{Q}^{s_j-2m,(s_j-2m)/(2b)}
\quad\mbox{for each}\quad j\in\{0,1\}
\end{equation}
(see also the book \cite[Theorem~5.7]{ZhitarashuEidelman98}).
Let $\psi$ be the interpolation parameter~\eqref{8f16}. Then the restriction of the operator
\eqref{16f49} with $j=0$ to the space
\begin{equation*}
\bigl[H^{s_0,s_{0}/(2b)}(\Omega),
H^{s_1,s_{1}/(2b)}(\Omega)\bigr]_{\psi}=
H^{s,s/(2b);\varphi}(\Omega)
\end{equation*}
is an isomorphism
\begin{align}\label{16f50}
\Lambda:\,H^{s,s/(2b);\varphi}(\Omega)\leftrightarrow
\bigl[&\mathcal{Q}^{s_0-2m,(s_0-2m)/(2b)},
\mathcal{Q}^{s_1-2m,(s_1-2m)/(2b)}\bigr]_{\psi}\\ \notag
&=\mathcal{Q}^{{s-2m,(s-2m)/(2b)};\varphi}.
\end{align}
Here, the equalities of spaces hold true up to equivalence of norms due to Proposition~\ref{8prop5} and Lemma~\ref{16lem7.4}. The operator \eqref{16f50} is an extension by continuity
of the mapping \eqref{16f4} because $C^{\infty}(\overline{\Omega})$
is dense in $H^{s,s/(2b);\varphi}(\Omega)$.
Thus, Theorem \ref{16th4.1} is proved in the case considered.

Examine now the $s\in E$ case. Choose $\varepsilon\in(0,1/2)$ arbitrarily.
Since $s\pm\varepsilon\notin E$ and $s-\varepsilon>\sigma_0$, we have the isomorphisms
\begin{equation}\label{8f37}
\Lambda:H^{s\pm\varepsilon,(s\pm\varepsilon)/(2b);\varphi}(\Omega)\leftrightarrow
\mathcal{Q}^{s\pm\varepsilon-2m,(s\pm\varepsilon-2m)/(2b);\varphi},
\end{equation}
as has just been proved.
They imply that the mapping \eqref{16f4} extends uniquely (by continuity) to an isomorphism
\begin{equation}\label{8f38}
\begin{aligned}
\Lambda:&
\bigl[H^{s-\varepsilon,(s-\varepsilon)/(2b);\varphi}(\Omega),
H^{s+\varepsilon,(s+\varepsilon)/(2b);\varphi}(\Omega)\bigr]_{1/2}\\
&\leftrightarrow
\bigl[\mathcal{Q}^{s-\varepsilon-2m,(s-\varepsilon-2m)/(2b);\varphi},
\mathcal{Q}^{s+\varepsilon-2m,(s+\varepsilon-2m)/(2b);\varphi}\bigr]_{1/2}=
\mathcal{Q}^{s-2m,(s-2m)/(2b);\varphi}.
\end{aligned}
\end{equation}
Recall that the last equality is the definition of the space $\mathcal{Q}^{s-2m,(s-2m)/(2b);\varphi}$. To complete the proof, it remains to apply Lemma~\ref{16lem7.5} for $W=\Omega$ to \eqref{8f38}.
\end{proof}

\begin{remark}\label{16rem8.1}
Let $s\in E$. The space $\mathcal{Q}^{s-2m,(s-2m)/(2b);\varphi}$ defined by formula \eqref{16f10} is
independent of the choice of the number $\varepsilon\in(0,1/2)$ up to equivalence of norms. Indeed, according to Theorem \ref{16th4.1} we have the isomorphism
\begin{equation*}
\Lambda:H^{s,s/(2b);\varphi}(\Omega)\leftrightarrow
\bigl[\mathcal{Q}^{s-\varepsilon-2m,(s-\varepsilon-2m)/(2b);\varphi},
\mathcal{Q}^{s+\varepsilon-2m,(s+\varepsilon-2m)/(2b);\varphi}\bigr]_{1/2}.
\end{equation*}
whenever $0<\varepsilon<1/2$. This directly implies the mentioned independence. Besides, the space $\mathcal{Q}^{s-2m,(s-2m)/(2b);\varphi}$ is embedded continuously in $\mathcal{H}^{s-2m,(s-2m)/(2b);\varphi}$. Indeed, choosing $\varepsilon\in(0,1/2)$, we get the continuous embeddings
\begin{equation*}
\mathcal{Q}^{s\mp\varepsilon-2m,(s\mp\varepsilon-2m)/(2b);\varphi}
\hookrightarrow
\mathcal{H}^{s\mp\varepsilon-2m,(s\mp\varepsilon-2m)/(2b);\varphi}
\end{equation*}
in view of $s\mp\varepsilon\in(0,\sigma_{0})\setminus E$ and the definition of the left-hand space. It follows from this that the embedding operator acts continuously from \eqref{16f10} to \eqref{16f67}, as stated.
\end{remark}

\begin{proof}[Proof of Theorem $\ref{16th4.3}$.]
We will first prove that, under its hypotheses
\eqref{16f13}--\eqref{16f15}, the implication
\begin{equation}\label{16f54}
u\in H^{s-\lambda,(s-\lambda)/(2b);\varphi}_{\mathrm{loc}}
(\Omega_0,\Omega')\;\Longrightarrow\;u\in H^{s-\lambda+1,(s-\lambda+1)/(2b);\varphi}_{\mathrm{loc}}
(\Omega_0,\Omega')
\end{equation}
holds for each integer $\lambda\geq1$ subject to $s-\lambda+1>\sigma_{0}$.

We arbitrarily choose a function $\chi\in C^\infty(\overline\Omega)$ with $\mbox{supp}\,\chi\subset\Omega_0\cup\Omega'$. For $\chi$ there exists a function $\eta\in C^\infty(\overline\Omega)$ such that $\mbox{supp}\,\eta\subset\Omega_0\cup\Omega'$ and
$\eta=1$ in a neighbourhood of $\mbox{supp}\,\chi$. Interchanging each of the differential operators
$A$,  $B_{j}$ and $\partial^k_t$
with the operator of the multiplication by $\chi$, we can write
\begin{equation}\label{16f55}
\begin{aligned}
\Lambda(\chi u)&=\Lambda(\chi\eta u)=\chi\,\Lambda(\eta u)+ \Lambda'(\eta u)\\
&=\chi\,\Lambda u+\Lambda'(\eta u)=\chi\,(f,g_{1},\dots,g_{m},h_{0},\dots,h_{\varkappa-1})+\Lambda'(\eta u).
\end{aligned}
\end{equation}
Here, $\Lambda':=(A',B'_{1},\ldots,B'_{m},C'_{0},\dots C'_{\varkappa-1})$ is an operator whose components act on every function $w(x,t)$ from $H^{\sigma_{0}-1,(\sigma_{0}-1)/(2b)}(\Omega)$ as follows:
\begin{gather}\label{16f56}
A'(x,t,D_x,\partial_t)w(x,t)=\sum_{|\alpha|+2b\beta\leq 2m-1}a^{\alpha,\beta}_{1}(x,t)\,D^\alpha_x\partial^\beta_t w(x,t),\\
B_{j}'(x,t,D_x,\partial_t)w(x,t)=\sum_{|\alpha|+2b\beta\leq m_j-1}
b_{j,1}^{\alpha,\beta}(x,t)\,D^\alpha_x\partial^\beta_t
w(x,t)\!\upharpoonright\!S,\quad j=1,\ldots,m\label{16f57},
\end{gather}
and
\begin{equation}\label{16f58}
C_{0}'w=0,\quad
C_{k}'(x,\partial_t)w=\sum_{l=0}^{k-1}
c_{l,k}(x)\,(\partial^{l}_t w)(x,0),\quad k=1,\ldots,\varkappa-1,
\end{equation}
where all $a^{\alpha,\beta}_{1}\in C^{\infty}(\overline{\Omega})$,
$b_{j,1}^{\alpha,\beta}\in C^{\infty}(\overline{S})$ and $c_{l,\,k}\in C^{\infty}(\overline{G})$.
This operator acts continuously between the spaces
\begin{equation}\label{16f59}
\Lambda':H^{\sigma,\sigma/(2b);\varphi}(\Omega)\rightarrow
\mathcal{H}^{\sigma+1-2m,(\sigma+1-2m)/(2b);\varphi}
\end{equation}
for every $\sigma>\sigma_{0}-1$. In the $\varphi(\cdot)\equiv1$ case, this follows directly from \eqref{16f56}, \eqref{16f57}, \eqref{16f58},
and the known properties of partial differential operators and trace operators on anisotropic Sobolev spaces (see, e.g., \cite[Chapter~I, Lemma~4, and Chapter~II, Theorems~3 and~7]{Slobodetskii58}). Remark only that each operator \eqref{16f58}, with $1\leq k\leq\varkappa-1$, acts continuously between the spaces $H^{\sigma,\sigma/(2b)}(\Omega)$ and $H^{\sigma-2b(k-1)-b}(G)\hookrightarrow H^{\sigma+1-2bk-b}(G)$.
The boundedness of the operator \eqref{16f59} in the general situation is plainly deduced from this case with the help of the Propositions~\ref{8prop4} and~\ref{8prop5} (see also \eqref{16f45} for $s:=\sigma+1$).

Owing to \eqref{16f13}, \eqref{16f14}, and \eqref{16f15}, we obtain the inclusion
$$
\chi\,(f,g_{1},\dots,g_{m},h_0,\dots,h_{\varkappa-1})
\in\mathcal{H}^{s-2m,(s-2m)/(2b);\varphi}.
$$
Besides, according to \eqref{16f59} with $\sigma:=s-\lambda$, we have the implication
\begin{equation*}
u\in H^{s-\lambda,(s-\lambda)/(2b);\varphi}_{\mathrm{loc}}
(\Omega_0,\Omega')
\Longrightarrow\Lambda'(\eta u)\in
\mathcal{H}^{s-\lambda+1-2m,(s-\lambda+1-2m)/(2b);\varphi}.
\end{equation*}
Hence, using \eqref{16f55}, we conclude that
\begin{equation}\label{16f60}
u\in H^{s-\lambda,(s-\lambda)/(2b);\varphi}_{\mathrm{loc}}
(\Omega_0,\Omega')
\Longrightarrow\Lambda(\chi u)\in
\mathcal{H}^{s-\lambda+1-2m,(s-\lambda+1-2m)/(2b);\varphi}.
\end{equation}

To deduce the required property \eqref{16f54} from \eqref{16f60} let us prove that
\begin{equation}\label{16f61}
\Lambda(\chi u)\in
\mathcal{H}^{\sigma-2m,(\sigma-2m)/(2b);\varphi}\Longrightarrow
\Lambda(\chi u)\in\mathcal{Q}^{\sigma-2m,(\sigma-2m)/(2b);\varphi}
\end{equation}
for every $\sigma>\sigma_0$. Assume that the premise of this implication is true for some $\sigma>\sigma_0$. Since $\mathrm{dist}(\mathrm{supp}\,\chi,\Gamma)>0$, we have the equality
$\Lambda(\chi u)=0$ near $\Gamma$. Hence, the vector $\Lambda(\chi u)$ satisfies the compatibility conditions \eqref{16f8} in which $(f,g_{1},\dots,g_{m},h_{0},\dots,h_{\varkappa-1})$ means $\Lambda(\chi u)$ and $\sigma$ is taken instead of $s$. This yields \eqref{16f61} in the $\sigma\notin E$ case due to the definition of the space $\mathcal{Q}^{\sigma-2m,(\sigma-2m)/(2b);\varphi}$.

In the opposite case of $\sigma\in E$, this space is defined by the interpolation. Considering this case, we choose a function $\chi_1\in C^{\infty}(\overline\Omega)$ such that $\chi_1=0$ in a neighbourhood of $\Gamma$ and that $\chi_1=1$ in a neighbourhood of $\mathrm{supp}\,\chi$. The mapping
$M_{\chi_1}:F\mapsto\chi_1 F$ acts continuously between the spaces
\begin{equation}\label{16f63}
M_{\chi_1}:
\mathcal{H}^{\sigma\pm\varepsilon-2m,(\sigma\pm\varepsilon-2m)/(2b);\varphi}
\to\mathcal{Q}^{\sigma\pm\varepsilon-2m,(\sigma\pm\varepsilon-2m)/(2b);\varphi}
\end{equation}
whenever $0<\varepsilon<1/2$ because the vector $\chi_1 F$ satisfies the compatibility conditions \eqref{16f8} in which $(f,g_{1},\dots,g_{m},h_{0},\dots,h_{\varkappa-1})$ means $\chi_1 F$ and the numbers $\sigma\pm\varepsilon\notin E$ are taken instead of~$s$. Applying the interpolation with the number parameter $1/2$ to \eqref{16f63}, we obtain a bounded operator
\begin{equation}\label{16f64}
\begin{aligned}
M_{\chi_1}&:
\bigl[\mathcal{H}^{\sigma-\varepsilon-2m,(\sigma-\varepsilon-2m)/(2b);\varphi},
\mathcal{H}^{\sigma+\varepsilon-2m,(\sigma+\varepsilon-2m)/(2b);\varphi}\bigr]_{1/2}
\\
&\to\bigl[\mathcal{Q}^{\sigma-\varepsilon-2m,(\sigma-\varepsilon-2m)/(2b);\varphi},
\mathcal{Q}^{\sigma+\varepsilon-2m,(\sigma+\varepsilon-2m)/(2b);\varphi}\bigr]_{1/2}.
\end{aligned}
\end{equation}
According to the interpolation formulas \eqref{16f67} and
\eqref{16f10}, this operator acts between the spaces
\begin{equation}\label{16f65}
M_{\chi_1}:\mathcal{H}^{\sigma-2m,(\sigma-2m)/(2b);\varphi}\to
\mathcal{Q}^{\sigma-2m,(\sigma-2m)/(2b);\varphi}.
\end{equation}
Owing to our choice of $\chi_1$, we have $\chi_1\Lambda(\chi u)=\Lambda(\chi u)$. Since $\Lambda(\chi u)\in
\mathcal{H}^{\sigma-2m,(\sigma-2m)/(2b);\varphi}$ by our assumption, we conclude
that
$$
\Lambda(\chi u)=\chi_1\Lambda(\chi u)\in \mathcal{Q}^{\sigma-2m,(\sigma-2m)/(2b);\varphi}
$$
due to \eqref{16f65}. Thus, the implication \eqref{16f61} is proved.

Now, using properties \eqref{16f60}, \eqref{16f61} with $\sigma:=s-\lambda+1$,
and Corollary~\ref{16cor4.2}, we conclude that
\begin{align*}
u\in H^{s-\lambda,(s-\lambda)/(2b);\varphi}_{\mathrm{loc}}
(\Omega_0,\Omega')
&\Longrightarrow\Lambda(\chi u)\in
\mathcal{H}^{s-\lambda+1-2m,(s-\lambda+1-2m)/(2b);\varphi}\\
&\Longrightarrow \Lambda(\chi u)\in
\mathcal{Q}^{s-\lambda+1-2m,(s-\lambda+1-2m)/(2b);\varphi}\\
&\Longrightarrow \chi u\in
H^{s-\lambda+1,(s-\lambda+1)/(2b);\varphi}(\Omega)
\end{align*}
for every $\chi\in C^\infty(\overline\Omega)$ subject to  $\mbox{supp}\,\chi\subset\Omega_0\cup\Omega'$. Note that Corollary~\ref{16cor4.2} is applicable here because $\chi u\in H^{\sigma_0,\sigma_0/(2b)}(\Omega)$ by the hypothesis of the theorem and because $s-\lambda+1>\sigma_{0}$. Thus, we have proved the required implication \eqref{16f54}.

Let us use this implication to prove the theorem, i.e. to show that $u\in H^{s,s/(2b);\varphi}_{\mathrm{loc}}(\Omega_0,\Omega')$. We separately examine the case of
$s\notin\mathbb{Z}$ and the case of $s\in\mathbb{Z}$.

Consider first the case of  $s\notin\mathbb{Z}$.
In this case, there exists an integer $\lambda_{0}\geq1$ such that
\begin{equation}\label{16f66}
s-\lambda_{0}<\sigma_{0}<s-\lambda_{0}+1.
\end{equation}
Using the implication \eqref{16f54} successively for
$\lambda:=\lambda_{0}$,
$\lambda:=\lambda_{0}-1$,..., $\lambda:=1$, we conclude that
\begin{align*}
u&\in H^{\sigma_0,\sigma_0/(2b)}(\Omega)\subset
H^{s-\lambda_{0},(s-\lambda_{0})/(2b);\varphi}_{\mathrm{loc}}(\Omega_0,\Omega')\\
&\Longrightarrow u\in
H^{s-\lambda_{0}+1,(s-\lambda_{0}+1)/(2b);\varphi}_{\mathrm{loc}}(\Omega_0,\Omega')
\Longrightarrow\ldots\Longrightarrow
u\in H^{s,s/(2b);\varphi}_{\mathrm{loc}}(\Omega_0,\Omega').
\end{align*}
Note that $u\in H^{\sigma_0,\sigma_0/(2b)}(\Omega)$ by the hypothesis of the theorem.

Consider now the case of $s\in\mathbb{Z}$.
In this case, there is no integer $\lambda_{0}$ that satisfies~\eqref{16f66}.
Nevertheless, since $s-\varepsilon\notin\mathbb{Z}$ and $s-\varepsilon>\sigma_{0}$ whenever $0<\varepsilon<1$, the inclusion $u\in H^{s-\varepsilon,(s-\varepsilon)/(2b);\varphi}_{\mathrm{loc}}(\Omega_0,\Omega')$
holds true as we have just proved. Hence,
using \eqref{16f54} with $\lambda:=1$, we conclude that
\begin{align*}
u\in H^{s-\varepsilon,(s-\varepsilon)/(2b);\varphi}_{\mathrm{loc}}(\Omega_0,\Omega')\subset
H^{s-1,(s-1)/(2b);\varphi}_{\mathrm{loc}}(\Omega_0,\Omega')
\Longrightarrow u\in H^{s,s/(2b);\varphi}_{\mathrm{loc}}(\Omega_0,\Omega').
\end{align*}
\end{proof}

\begin{proof}[Proof of Theorem $\ref{16th4.4}$.]
Choosing a sufficiently small number $\varepsilon>0$, we put $U_{\varepsilon}:=\{x\in U:\mathrm{dist}(x,\partial U)>\varepsilon\}$, $\Omega_{\varepsilon}:=U_{\varepsilon}\cap\Omega$, and $\Omega'_{\varepsilon}:=U_{\varepsilon}\cap\partial\overline{\Omega}$. Consider a function $\chi_{\varepsilon}\in C^\infty(\overline\Omega)$ such that $\mbox{supp}\,\chi_{\varepsilon}\subset\Omega_0\cup\Omega'$ and that $\chi_{\varepsilon}=1$ on $\Omega_{\varepsilon}\cup\Omega'_{\varepsilon}$. Owing to Theorem~\ref{16th4.3}, we have the inclusion $\chi_{\varepsilon}u\in H^{s,s/(2b);\varphi}(\Omega)$ were $s=p+b+n/2$ and $\varphi$ satisfies condition \eqref{9f4.7}. Hence, there exists a distribution $w_{\varepsilon}\in H^{s,s/(2b);\varphi}(\mathbb{R}^{n+1})$ such that $w_{\varepsilon}=\chi_{\varepsilon}u=u$ on $\Omega_{\varepsilon}$. Let the indices $\alpha=(\alpha_{1},\ldots,\alpha_{n})$ and $\beta$ satisfy the condition $|\alpha|+2b\beta\leq p$. Then, according to  \cite[Lemma~8.1(i)]{LosMikhailetsMurach17CPAA}, the generalized partial derivative $D_{x}^{\alpha}\partial_{t}^{\beta}w_{\varepsilon}(x,t)$ is continuous on $\mathbb{R}^{n+1}$. Hence, the distribution  $v(x,t):=D_{x}^{\alpha}\partial_{t}^{\beta}u(x,t)$ in $\Omega$ is continuous on $\Omega_{\varepsilon}\cup\Omega'_{\varepsilon}$; i.e.,
\begin{equation*}
v(\omega)=\int\limits_{\Omega_\varepsilon}v_{\varepsilon}(x,t)\,
\omega(x,t)\,dxdt
\end{equation*}
for every test function $\omega\in C^{\infty}(\Omega)$ subject to $\mathrm{supp}\,\omega\subset\Omega_\varepsilon$, with $v_{\varepsilon}$ denoting the continuous function $v_{\varepsilon}(x,t):=
D_{x}^{\alpha}\partial_{t}^{\beta}w_{\varepsilon}(x,t)$ of $(x,t)\in\Omega_{\varepsilon}\cup\Omega'_{\varepsilon}$. We define the continuous function $v_{0}$ on $\Omega_{0}\cup\Omega'_{0}$ by the formula $v_{0}:=v_{\varepsilon}$ on $\Omega_{\varepsilon}\cup\Omega'_{\varepsilon}$ whenever $0<\varepsilon\ll1$. This function is well defined because $0<\delta<\varepsilon$ implies that $v_{\delta}=v_{\varepsilon}$ on $\Omega_{\varepsilon}\cup\Omega'_{\varepsilon}$. Then $v$ satisfies \eqref{rem-to-th4.4} for every test function $\omega\in C^{\infty}(\Omega)$ with $\mathrm{supp}\,\omega\subset\Omega_0$ because
$\mathrm{supp}\,\omega\subset\Omega_\varepsilon$ for a sufficiently small number $\varepsilon>0$ depending on $\omega$.
\end{proof}

Ending this section, we  substantiate Remark~\ref{16rem4.5}. Let $\varphi\in\mathcal{M}$, and let an integer $p\geq0$ be subject to the condition  $s:=p+b+n/2>\sigma_{0}$. Assume that every function $u\in\nobreak H^{\sigma_0,\sigma_0/(2b)}(\Omega)$ satisfies \eqref{f-rem4.5} and show that $\varphi$ then satisfies \eqref{9f4.7}. Let $V$ be a nonempty open subset of $\mathbb{R}^{n+1}$ such that $\overline{V}\subset\Omega_0$. We arbitrarily choose a function $w\in H^{s,s/(2b);\varphi}(\mathbb{R}^{n+1})$ such that $\mathrm{supp}\,w\subset V$. Put $u:=w\!\upharpoonright\!\Omega\in H^{s,s/(2b);\varphi}(\Omega)$ and
\begin{equation*}
(f,g_{1},...,g_{m},h_0,\dots,h_{\varkappa-1}):=\Lambda u\in
\mathcal{H}^{s-2m,(s-2m)/(2b);\varphi}.
\end{equation*}
The function $u$ satisfies the premise of the implication \eqref{f-rem4.5}. Hence,   $u$ satisfies the conclusion of Theorem~\ref{16th4.4} due to our assumption. Thus, the generalized derivative $D_{x}^{\alpha}\partial_{t}^{\beta}u(x,t)$ is continuous on $\Omega_0\cup\Omega'$ whenever $|\alpha|+2b\beta\leq p$. Specifically, each derivative $D_{1}^{j}u$, with $0\leq j\leq p$, is continuous on $V$. Therefore, each derivative $D_{1}^{j}w$, with $0\leq j\leq p$, is continuous on $\mathbb{R}^{n+1}$. Hence, $\varphi$ satisfies \eqref{9f4.7} due to \cite[Lemma~8.1(ii)]{LosMikhailetsMurach17CPAA}.  Remark~\ref{16rem4.5} is substantiated.

\section*{Appendix}

Along with the explicit compatibility conditions \eqref{16f8}, other (and less explicit) forms of them are often used in the theory of general parabolic initial-boundary value problems (see, e.g., \cite{AgranovichVishik64, LionsMagenes72ii, ZhitarashuEidelman98}). Our proof of Theorem~\ref{16th4.1} is based on the isomorphism theorem obtained in works \cite{Zhitarashu85, ZhitarashuEidelman98}. They refer to the compatibility conditions introduced in \cite[\S\,11]{AgranovichVishik64}. These conditions are equivalent to \eqref{16f8} on some assumptions about $s$, which is considered known. But we have not found the proof of this fact in the literature. Therefore, we prefer to give the proof for the sake of completeness of the presentation.

Let us formulate the compatibility conditions given in  \cite[\S\,11]{AgranovichVishik64}. They use some function spaces, which we introduce now.
Let $V$ be an open nonempty subset of $\mathbb{R}^{k}$, with $2\leq k\in\mathbb{Z}$, and let $s>0$.
The linear space $H^{s,s/(2b)}_{+}(V)$ is defined to consist of the restrictions
$u=w\!\upharpoonright\!V$ of all functions $w\in H^{s,s/(2b)}(\mathbb{R}^{k})$ which
vanish whenever $t<0$, we considering $w$ as a function $w(x,t)$ of
$x=(x_{1},\ldots,x_{k-1})\in\mathbb{R}^{k-1}$ and $t\in\mathbb{R}$.
The space is endowed with the norm
\begin{align*}
&\|u\|_{H^{s,s/(2b)}_{+}(V)} \\
&:=\inf\bigl\{\|w\|_{H^{s,s/(2b)}(\mathbb{R}^{k})}:\,w\in
H^{s,s/(2b)}(\mathbb{R}^{k}),\;w(x,t)=0\;\,\mbox{if}\;\,t<0,
\;u=w\!\upharpoonright\!V\bigr\}.
\end{align*}
We need this Hilbert space in the case were $V=\Omega$ and $k=n+1$ or
in the case were $V=\Pi$ and $k=n$, with $\Pi:=\mathbb{R}^{n-1}\times(0,\tau)$.
Changing $H^{s,s\gamma;\varphi}(\Pi)$ for $H^{s,s\gamma}_{+}(\Pi)$ in the definition of $H^{s,s\gamma;\varphi}(S)$, we define the Hilbert space $H^{s,s\gamma}_{+}(S)$. This space does not depend up to equivalence of norms on our choice of local charts and partition of unity on $\Gamma$ \cite[Lemma~3.1]{LosMikhailetsMurach17CPAA}.

Assume that $s\geq\sigma_0$, $s\notin E$, and $s/(2b)+1/2\notin\mathbb{Z}$.
A vector
\begin{equation}\label{vector-F}
F:=\bigl(f,g_1,...,g_m,h_0,...,h_{\varkappa-1}\bigr)\in
\mathcal{H}^{s-2m,(s-2m)/(2b)}
\end{equation}
satisfies the compatibility conditions in the sense of \cite[\S\,11]{AgranovichVishik64} if there exists a function
$v=v(x,t)$ from
$H^{s,s/(2b)}(\Omega)$ such that
\begin{gather}\label{16f73}
f-Av\in H^{s-2m,(s-2m)/(2b)}_{+}(\Omega), \\ \label{16f74}
g_j-B_jv\in H^{s-m_j-1/2,(s-m_j-1/2)/(2b)}_{+}(S)
\quad\mbox{for each}\quad j\in\{1,\dots,m\},\\ \label{16f75}
h_k=\partial^{k}_{t}v\big|_{t=0}\quad\mbox{for each}\quad
k\in\{0,\dots,\varkappa-1\}.
\end{gather}
(Note that $s+1/2\notin\mathbb{Z}\Rightarrow s\notin E$. Hence, the restrictions put on $s_{j}$ at the beginning of our proof of Theorem~\ref{16th4.1} and caused by the use of \cite[Theorem~9.1]{Zhitarashu85} are somewhat stronger than the assumptions just made about $s$.)

Let us prove that the collection of these compatibility conditions is
equivalent to \eqref{16f8} for every vector \eqref{vector-F}. Owing to
\cite[Lemma 5.1]{LosMikhailetsMurach17CPAA}, we rewrite conditions
\eqref{16f73} and \eqref{16f74} in the following equivalent form:
\begin{equation}\label{16f76}
\begin{aligned}
\partial^l_t&(f-Av)(x,0)=0\quad\mbox{for almost all}\quad x\in G\\
&\mbox{whenever}\quad l\in\mathbb{Z}
\quad\mbox{and}\quad 0\leq l<(s-2m)/(2b)-1/2,
\end{aligned}
\end{equation}
and
\begin{equation}\label{16f77}
\begin{aligned}
\partial^k_t &(g_j-B_jv)(x,0)=0\;\;\mbox{for almost all}\;\; x\in \Gamma
\;\;\mbox{and for each}\;\;j\in\{1,\dots,m\}\\
&\mbox{whenever}\;\; k\in\mathbb{Z}
\;\;\mbox{and}\;\; 0\leq k<(s-m_j-1/2)/(2b)-1/2.
\end{aligned}
\end{equation}
The collection of conditions \eqref{16f75} and \eqref{16f76} is equivalent to \eqref{16f9} where $0\leq k<s/(2b)-1/2$ provided that we put
\begin{equation}\label{functions-v}
v_k(x)=(\partial^{k}_{t}v)(x,0)\;\;\mbox{for almost all}\;\;x\in G\;\;
\mbox{whenever}\;\;0\leq k<s/(2b)-1/2.
\end{equation}
Indeed, note first that \eqref{16f76} is equivalent to the collection of conditions
\begin{align*}
\partial^{k-\varkappa}_t&((a^{(0,\ldots,0),\varkappa})^{-1}(f-Av))(x,0)=0
\quad\mbox{for almost all}\quad x\in G\\
&\mbox{whenever}\quad k\in\mathbb{Z}
\quad\mbox{and}\quad \varkappa\leq k<s/(2b)-1/2.
\end{align*}
If $\varkappa\leq k<s/(2b)-1/2$, then
\begin{equation*}
-\biggl(\partial^{k-\varkappa}_{t}
\frac{Av}{a^{(0,\ldots,0),\varkappa}}\biggr)(x,0)=
\sum_{|\alpha|+2b\beta\leq 2m}\,
\sum\limits_{q=0}^{k-\varkappa}
\binom{k-\varkappa}{q}(\partial^{k-\varkappa-q}_t
a_{0}^{\alpha,\beta})(x,0)\,D^\alpha_{x}v_{\beta+q}(x)
\end{equation*}
for almost all $x\in G$. Note that $v_{k}$ presents only in the summand corresponding to $\nobreak{\alpha=(0,\ldots,0)}$, $\beta=\varkappa$, and $q=k-\varkappa$ and that this summand equals $-v_{k}$. Hence,
\begin{align*}
v_{k}(x)=\biggl(&\partial^{k-\varkappa}_{t}
\frac{Av}{a^{(0,\ldots,0),\varkappa}}\biggr)(x,0)\\
+&\sum_{\substack{|\alpha|+2b\beta\leq 2m,\\ \beta\leq\varkappa-1}}\,
\sum\limits_{q=0}^{k-\varkappa}
\binom{k-\varkappa}{q}(\partial^{k-\varkappa-q}_t
a_{0}^{\alpha,\beta})(x,0)\,D^\alpha_{x}v_{\beta+q}(x).
\end{align*}
It is evident now that \eqref{16f76} is equivalent to the collection of relations \eqref{16f9} where $\varkappa\leq k<s/(2b)-1/2$. This gives the required equivalence of \eqref{16f75} and \eqref{16f76} to \eqref{16f9}.

Suppose now that a vector \eqref{vector-F} satisfies the compatibility conditions \eqref{16f8} in which $v_{k}$ and $B_{j,k}$ are defined by \eqref{16f9} and \eqref{16f9B}. Let us define a vector
$$
V=(v_0,...,v_r)\in \bigoplus_{k=0}^{r}H^{s-2bk-b}(G),
$$
by formulas \eqref{16f9}, with $r:=[s/(2b)-1/2]$.
Owing to \cite[Chapter~2, Theorem~10]{Slobodetskii58} there exists
a function $v\in H^{s,s/(2b)}(\Omega)$ that satisfies \eqref{functions-v}. It follows from this and \eqref{16f2} and \eqref{16f9B} that
\begin{equation}\label{last-formula}
\begin{aligned}
\partial^k_{t}&B_jv\big|_{t=0}=B_{j,k}(v_0,\dots,v_{[m_j/(2b)]+k})
\quad\mbox{on}\;\;G\\
&\mbox{for each}\;\;j\in\{1,\dots,m\}\;\;\mbox{and}\;\;k\in\mathbb{Z}
\;\;\mbox{such that}\;\;0\leq k<\frac{s-m_j-1/2-b}{2b}.
\end{aligned}
\end{equation}
Hence, \eqref{16f8} implies \eqref{16f77}. Besides, \eqref{16f9} implies \eqref{16f75} and \eqref{16f76}, as we have mentioned. Thus, the vector \eqref{vector-F} satisfies conditions \eqref{16f73}--\eqref{16f75}.

Conversely, suppose that a vector \eqref{vector-F} satisfies relations \eqref{16f73}--\eqref{16f75} for a certain function $v\in H^{s,s/(2b)}(\Omega)$. This implies \eqref{16f9} and \eqref{16f77} provided that we define the functions $v_k$ by \eqref{functions-v}. Then the compatibility conditions \eqref{16f8} follow from \eqref{16f9} and \eqref{16f77} in view of \eqref{last-formula}.


\begin{thebibliography}{99}

\bibitem{AgranovichVishik64}
M. S. Agranovich,  M. I. Vishik,
Elliptic problems with parameter and parabolic problems of general form,
Uspehi Mat. Nauk \textbf{19} (1964), 53--161 (Russian). [English translation in Russian Math. Surveys \textbf{19} (1964), 53--157.]

\bibitem{AnopKasirenko16MFAT}
A. V. Anop, T. M. Kasirenko, Elliptic boundary-value problems in H\"ormander spaces, Methods Funct. Anal. Topology \textbf{22}
(2016), no.~4, 295--310.

\bibitem{AnopMurach14UMJ}
A. V. Anop,  A. A. Murach, Regular elliptic boundary-value problems in the extended Sobolev scale, Ukrainian Math. J. \textbf{66} (2014), no.~7,
969--985.

\bibitem{Berezansky68}
Yu. M. Berezansky, Expansions in Eigenfunctions of Selfadjoint Operators,
Transl. Math. Monogr. \textbf{17}, American Mathematical Society, Providence, RI, 1968.

\bibitem{BerghLefstrem76}
J. Bergh, J. L\"ofstr\"om, Interpolation Spaces, Springer, Berlin, 1976.

\bibitem{BesovIlinNikolskii75}
O. V. Besov, V. P. Il’in, S. M. Nikol’skii, Integral Representations of Functions and Embedding Theorems, Nauka, Moscow, 1975 (Russian).

\bibitem{BinghamGoldieTeugels89}
N. H. Bingham, C. M. Goldie, J. L. Teugels, Regular Variation,
Encyclopedia Math. Appl. \textbf{27}, Cambridge University Press, Cambridge, 1989.

\bibitem{BuldyginIndlekoferKlesovSteinebach18}
V. V. Buldygin, K.-H. Indlekofer, O. I. Klesov, J. G. Steinebach,
Pseudo-Regularly Varying Functions and Generalized Renewal Processes, Probab. Theory Stochastic Modelling \textbf{91}, Springer, Cham, 2018.

\bibitem{DenkHieberPruss07}
R. Denk, M. Hieber, J. Pr\"uss, Optimal $L^p$-$L^q$-estimates for parabolic boundary value problems with inhomogeneous data, Math. Z. \textbf{257} (2007), no. 1, 193--224.

\bibitem{Eidelman64}
S. D. Eidel'man, Parabolic Systems, North-Holland, Amsterdam, 1969.

\bibitem{Eidelman94}
S. D. Eidel'man, Parabolic equations, Encycl. Math. Sci. \textbf{63}, Partial differential equations, VI, Springer, Berlin, 1994, 205--316.

\bibitem{ZhitarashuEidelman98}
S. D. Eidel'man, N. V. Zhitarashu, Parabolic Boundary Value Problems,
Oper. Theory Adv. Appl. \textbf{101}, Birkh\"auser, Basel, 1998.

\bibitem{FarkasLeopold06}
W. Farkas, H.-G. Leopold, Characterisations of function spaces of generalized smoothness, Ann. Mat. Pura Appl. \textbf{185} (2006), no~1, 1--62.

\bibitem{FoiasLions61}
C. Foia\c{s}, J.-L. Lions, Sur certains th\'eor\`emes d'interpolation,
Acta Scient. Math. Szeged \textbf{22} (1961), no.~3--4, 269--282.

\bibitem{Friedman64}
A. Friedman, Partial Differential Equations of Parabolic Type,
Prentice-Hall Inc., Englewood Cliffs, N.J., 1964.

\bibitem{Hermander63}
L. H\"ormander, Linear Partial Differential Operators,
Grundlehren Math. Wiss. \textbf{116}, Springer, Berlin, 1963.

\bibitem{Hermander83}
L. H\"ormander, The Analysis of Linear Partial Differential Operators, vol.~2, Differential Operators with Constant Coefficients, Grundlehren Math. Wiss. \textbf{257}, Springer, Berlin, 1983.

\bibitem{IlinKalashnikovOleinik62}
A. M. Il'in, A. S. Kalashnikov, O. A. Oleinik,
Linear equations of the second order of parabolic type,
Uspekhi Mat. Nauk \textbf{17} (1962), no~3, 3--146 (Russian).
[English translation in Russian Math. Surveys, \textbf{17:3} (1962), 1–143.]

\bibitem{Ilin60}
V. A. Il'in The solvability of mixed problems for hyperbolic and parabolic equations,
Uspekhi Mat. Nauk \textbf{15} (1960), no~2, 97--154 (Russian).
[English translation in Russian Math. Surveys, \textbf{15:1} (1960), 85–142.]

\bibitem{Ivasyshen90}
S. D. Ivasyshen, Green Matrices of Parabolic Boundary-Value Problems, Vyshcha Shkola, Kiev, 1990 (Russian).

\bibitem{Jacob010205}
N. Jacob, Pseudodifferential Operators and Markov Processes (in 3 volumes), Imperial College Press, London, 2001, 2002, 2005.

\bibitem{KalyabinLizorkin87}
G. A. Kalyabin, P. I. Lizorkin, Spaces of functions of generalized
smoothness, Math. Nachr. \textbf{133} (1987), 7--32.

\bibitem{Karamata30a}
J. Karamata, Sur certains “Tauberian theorems”\;de M.~M.~Hardy et Littlewood, Mathematica (Cluj) \textbf{3} (1930), 33--48.

\bibitem{KreinPetuninSemenov82}
S.~G. Krein, Yu.~I Petunin, E.~M. Sem\"enov, Interpolation of Linear Operators, Transl. Math. Monogr. \textbf{54}, American Mathematical Society, Providence, R.I., 1982.

\bibitem{LadyzhenskajaSolonnikovUraltzeva67}
O. A. Lady\v{z}enskaja, V. A. Solonnikov, N. N. Ural'tzeva,
Linear and Quasilinear Equations of Parabolic Type, Transl. Math. Monogr. \textbf{23}, American Mathematical Society, Providence, R.I., 1968.

\bibitem{LionsMagenes72i}
J.-L. Lions, E. Magenes, Non-Homogeneous Boundary-Value Problems and
Applications, vol.~1, Grundlehren Math. Wiss. \textbf{181}, Springer, Berlin, 1972.

\bibitem{LionsMagenes72ii}
J.-L. Lions, E. Magenes, Non-Homogeneous Boundary-Value Problems and
Applications, vol.~2, Grundlehren Math. Wiss. \textbf{182}, Springer, Berlin, 1972.

\bibitem{Los15UMG5}
V. M. Los, Mixed Problems for the two-dimensional heat-conduction
equation in anisotropic H\"ormander spaces, Ukranian Math. J. \textbf{67} (2015), no.~5, 735--747.

\bibitem{Los16JMathSci}
V. M. Los, Anisotropic H\"ormander spaces on the lateral surface of a cylinder, J. Math. Sci. (N. Y.) \textbf{217} (2016), no.~4, 456--467.

\bibitem{Los16UMG6}
V. M. Los, Theorems on isomorphisms for some parabolic initial-boundary-value problems in H\"ormander spaces: limiting case,
Ukranian Math. J. \textbf{68} (2016), no.~6, 894--909.

\bibitem{Los17UMG11}
V. M. Los, Sufficient conditions for the solutions of general parabolic initial-boundary-value problems to be classical, Ukranian Math. J. \textbf{68} (2017), no.~11, 1756--1766.

\bibitem{LosMikhailetsMurach17CPAA}
V. Los, V. A. Mikhailets, A. A. Murach, An isomorphism theorem for parabolic problems in H\"ormander spaces and its applications, Commun. Pure Appl. Anal. \textbf{16} (2017), no.~1, 69--97.

\bibitem{LosMurach13MFAT2}
V. Los, A. A. Murach, Parabolic problems and interpolation with a function parameter, Methods Funct. Anal. Topology \textbf{19} (2013), no.~2, 146--160.

\bibitem{LosMurach17OpenMath}
V. Los, A. Murach, Isomorphism theorems for some parabolic initial-boundary value problems in H\"ormander spaces,
Open Mathematics \textbf{15} (2017), 57--76.

\bibitem{Lunardi95}
A. Lunardi,  Analytic Semigroups and Optimal Regularity in Parabolic Problems, Birkhauser Verlag, Basel, 1995.

\bibitem{MikhailetsMurach12BJMA2}
V. A. Mikhailets, A. A. Murach, The refined Sobolev scale, inter\-po\-la\-tion, and elliptic problems, Banach J. Math. Anal.
\textbf{6} (2012), no.~2, 211--281.

\bibitem{MikhailetsMurach13UMJ3}
V. A. Mikhailets, A. A. Murach, Extended Sobolev scale and elliptic operators, Ukrainian Math.~J. \textbf{65} (2013), no.~3, 435--447.

\bibitem{MikhailetsMurach14}
V. A. Mikhailets, A. A. Murach, Ho\"rmander Spaces, Interpolation, and
Elliptic Problems, de Gruyter Stud. Math \textbf{60}, De Gruyter, Berlin, 2014.

\bibitem{MikhailetsMurach15ResMath1}
V. A. Mikhailets, A. A. Murach, Interpolation Hilbert spaces between Sobolev spaces, Results Math. \textbf{67} (2015), no.~1, 135--152.

\bibitem{NicolaRodino10}
F. Nicola, L. Rodino, Global Pseudodifferential Calculas on Euclidean spaces, Pseudo Diff. Oper. \textbf{4}, Birkh\"aser, Basel, 2010.

\bibitem{Paneah00}
B. Paneah, The Oblique Derivative Problem. The Poincar\'e problem, Wiley--VCH, Berlin, 2000.

\bibitem{Peetre66}
J. Peetre, On interpolation functions, Acta Sci. Math. (Szeged) \textbf{27} (1966), 167--171.

\bibitem{Peetre68}
J. Peetre, On interpolation functions II, Acta Sci. Math. (Szeged) \textbf{29} (1968), 91--92.

\bibitem{Petrovskii38}
I. G. Petrovskii, On the Cauchy problem for systems of partial differential equations in the domain of non-anallytic functions, Bull. Mosk. Univ., Mat. Mekh. \textbf{1} (1938), no.~7, 1--72 (Russian).

\bibitem{Roitberg96}
Ya. Roitberg, Elliptic Boundary Value Problems in the Spaces of
Distributions, Math. Appl. \textbf{384}, Kluwer Academic Publishers, Dordrecht, 1996.

\bibitem{Seneta76}
E. Seneta, Regularly Varying Functions, Lecture Notes in Math. \textbf{508}, Springer, Berlin, 1976.

\bibitem{Slenzak74}
G. Slenzak, Ellptic problems in a refined scale of spaces, Moscow Univ.
Math. Bull. \textbf{29} (1974), no. 3-4, 80--88.

\bibitem{Slobodetskii58}
L. N. Slobodeckii, Generalized Sobolev spaces and their application to boundary problems for partial differential equations,
(Russian), Leningrad. Gos. Ped. Inst. Uchen. Zap. \textbf{197} (1958), 54--112 [English translation in Amer. Math. Soc. Transl. (2) \textbf{57} (1966), 207--275].

\bibitem{Solonnikov64}
V. A. Solonnikov, Apriori estimates for solutions of second-order equations of parabolic type, Tr. Mat. Inst. Steklova \textbf{70} 1964, 133–-212 (Russian).

\bibitem{Stepanets05}
A. I. Stepanets, Methods of Approximation Theory, VSP, Utrecht, 2005.

\bibitem{Triebel83}
H. Triebel, Theory of Function Spaces, Monogr. Math. \textbf{78}, Birkh\"auser, Basel, 1983.

\bibitem{Triebel95}
H. Triebel, Interpolation Theory, Function Spaces, Differential Operators, 2nd ed., Johann Ambrosius Barth, Heidelberg, 1995.

\bibitem{Triebel01}
H. Triebel, The Structure of Functions, Birkh\"aser, Basel, 2001.

\bibitem{VolevichPaneah65}
L. R. Volevich, B. P. Paneah, Certain spaces of generalized functions and embedding theorems, Uspekhi Mat. Nauk \textbf{20} (1965), 3--74 (Russian). [English translation in Russian Math. Surveys \textbf{20} (1965), 1--73.]

\bibitem{Zagorskii61}
T. Ya. Zagorskii, A Mixed Problem for Systems of Partial Differential Equations of Parabolic Type, L'vov State University, L'vov, 1961 (Russian).

\bibitem{Zhitarashu85}
N. V. Zhitarashu, Theorems on complete collection of isomorphisms in the $L_2$-theory of generalized solutions for one equation parabolic in Petrovskii’s sense, Mat. Sb. \textbf{128} (1985), no. 4, 451--473.

\end{thebibliography}
\end{document}